\documentclass[11pt]{amsart}
\usepackage{amsfonts}
\usepackage{amsmath}
\usepackage{amssymb}
\usepackage{url}
\usepackage{multicol}
\usepackage{stmaryrd}
\usepackage{cite}
\usepackage{epsfig}
\usepackage{color}
\usepackage{graphics}
\usepackage{graphicx}
\usepackage{epstopdf}
\usepackage{multicol,graphics}
\newcommand\bes{\begin{eqnarray}}
\newcommand\ees{\end{eqnarray}}

\newtheorem{theorem}{Theorem}[section]
\newtheorem{lemma}[theorem]{Lemma}
\newtheorem{corollary}[theorem]{Corollary}
\newtheorem{claim}[theorem]{Claim}

\newtheorem{remark}[theorem]{Remark}

\numberwithin{equation}{section}
\allowdisplaybreaks

\begin{document}
\title[Time periodic nonlocal dispersal operators and applications]{\textbf{The generalised principal eigenvalue of time-periodic nonlocal dispersal operators and applications}}

\author[Y.H. Su, W.T. Li, Y. Lou and F.Y. Yang]{Yuan-Hang Su$^{1}$, Wan-Tong Li$^{1,*}$, Yuan Lou$^{2}$ ~and Fei-Ying Yang$^{1}$}
\thanks{\hspace{-.5cm}
$^1$ School of Mathematics and Statistics, Lanzhou University, Lanzhou, Gansu, 730000,
PRC 
\\
$^2$ Department of Mathematics, The Ohio State University, Columbus, Ohio 43210, USA
\\
$^*${\sf Corresponding author} (wtli@lzu.edu.cn)}

\date{\today}

\begin{abstract}
This paper is mainly concerned with the generalised principal eigenvalue
for time-periodic nonlocal dispersal operators.
We first establish the equivalence between two
different characterisations of the generalised principal eigenvalue.
We further investigate the dependence of the generalised principal eigenvalue on the frequency, the dispersal rate and the dispersal spread.
Finally, these qualitative results for time-periodic linear operators
are applied to time-periodic nonlinear KPP equations with nonlocal dispersal,
focusing on the effects of the frequency, the dispersal rate and the dispersal spread
on the existence and stability
of positive time-periodic solutions to  nonlinear equations.

\medskip

\textbf{Key words:} Time-periodic nonlocal dispersal operator; Generalised principal eigenvalue; KPP equation;
Positive time-periodic solution; Asymptotic behavior

\medskip

\textbf{AMS Subject Classification (2010):} 35K57; 35R09; 45C05; 47G20; 92D25.
\end{abstract}

\maketitle

\tableofcontents

\newpage

\section{Introduction}

\noindent

In recent years nonlocal dispersal evolution equations have been widely
used to model non-adjacent diffusive phenomena which exhibit long range
internal interactions; See \cite{Berestycki-2016-JMB,Fife-2003-Trends,
Hutson-2003-JMB,Kot-1996-Ecology,Murray-2002-IAM} and references therein. The principal eigenvalues of nonlocal dispersal
operators serve as a basic tool for the investigation of nonlocal dispersal
equations. Many studies have been devoted to the understanding  of the principal eigenvalues for elliptic-type nonlocal dispersal operators
and their qualitative properties; See \cite{Bates-2007-JMAA,Berestycki-2016-JFA,
Coville-2010-JDE,Coville-2013-Poincar,Coville-2013-AML,Rossi-2009-JDE,
Hutson-2003-JMB,
Kao-2010-DCDS,Li-2017-DCDS,Shen-2015-DCDS,Shen-2010-JDE,Su-2019-submitted,
Sun-2015-DCDS,Sun-2014-JDE,Yang-2016-DCDSA} and references therein. As far as time-periodic nonlocal dispersal operators are concerned, however, there
is less understanding for the associated principal eigenvalue, especially the dependence of the principal eigenvalue with respect to the underlying parameters. Similar to the time-periodic random dispersal operators
\cite{Hess-1991,Nadin-2009-AMPA,Peng-2015-CVPDE},
the principal eigenvalues for time-periodic nonlocal dispersal operators 
are relevant 
when a time-periodic
environment is involved.

In this paper, we are interested in the following time-periodic nonlocal
dispersal operators:
\begin{equation}\label{101}
L_{\Omega}^{\tau,\mu,\sigma,m}[v](x,t):= -\tau v_{t}(x,t)+ \frac{\mu}{\sigma^{m}}\bigg(\int_{\Omega}J_{\sigma}(x-y)v(y,t)dy
-h^{\sigma}(x)v(x,t)\bigg)+a(x,t)v(x,t),
\end{equation}
where $(x,t)\in\bar{\Omega}\times\mathbb{R}$, $\Omega\subset\mathbb{R}^N$ is a bounded domain, $\tau>0$ is the frequency, $\mu>0$ is the dispersal rate, $\sigma>0$ is the dispersal spread which characterises the dispersal range, $m\geq0$ is the cost parameter, $J_{\sigma}(\cdot)=\frac{1}{\sigma^{N}}J(\frac{\cdot}{\sigma})$ is the scaled dispersal kernel. Throughout the paper, we will make the following assumptions on the dispersal kernel $J$, a family of functions $\{h^{\sigma}\}_{\sigma>0}$ and function $a$:
\begin{itemize}
\item[($J$)] $J\in C(\mathbb{R}^N)$ is nonnegative symmetric with
compact support on the unit ball $B_{1}(0)$, $J(0)>0$ and $\int_{\mathbb{R}^N}J(z)dz=1$;
\item[($H$)] $h^{\sigma}\in C(\bar{\Omega})$ and
 there exists a constant $M>0$ such that
  $\|h^{\sigma}\|_{C(\bar{\Omega})}\leq M$ for all $\sigma>0$;
\item[($A$)] $a\in C_{1}(\bar{\Omega}\times\mathbb{R}):=  
\{v\in C(\bar{\Omega}\times\mathbb{R}) \ | \ v(x,t+1)= v(x,t), \ (x,t)\in \bar{\Omega}\times\mathbb{R}\}$.
\end{itemize}
Define the spaces $\chi_{\Omega}, \chi_{\Omega}^{+}, \chi_{\Omega}^{++}$ as follows:
\begin{equation*}
\begin{aligned}
&\chi_{\Omega}=\{v\in C^{0,1}(\bar{\Omega}\times\mathbb{R}) \ | \ v(x,t+1)= v(x,t), \ (x,t)\in \bar{\Omega}\times\mathbb{R}\},\\
&\chi_{\Omega}^{+}=\{v\in \chi_{\Omega}\ | \ v(x,t)\geq 0, \ (x,t)\in \bar{\Omega}\times\mathbb{R}\},\\
&\chi_{\Omega}^{++}=\{v\in \chi_{\Omega}\ | \ v(x,t)>0, \ (x,t)\in \bar{\Omega}\times\mathbb{R}\},
\end{aligned}
\end{equation*}
where $C^{0,1}(\bar{\Omega}\times\mathbb{R})$ denotes the class of functions that are continuous in $x$ and $C^{1}$ in $t$. The operator
$L_{\Omega}^{\tau,\mu,\sigma,m}$ is then considered as an unbounded linear operator on the space $C_{1}(\bar{\Omega}\times\mathbb{R})$ with domain $\chi_{\Omega}$ , namely,
\begin{equation*}
L_{\Omega}^{\tau,\mu,\sigma,m}: \chi_{\Omega}\ \subset C_{1}(\bar{\Omega}\times\mathbb{R})
\rightarrow C_{1}(\bar{\Omega}\times\mathbb{R}).
\end{equation*}

It may be worthwhile to point out that the time-periodic nonlocal dispersal operators
of the form \eqref{101} 
include several kinds of boundary conditions, such as Dirichlet, Neumann and  mixed type; 
See \cite{Rossi-2006-JMPA,Raul-2015-DCDS,Rawal-2012-JDDE}. 

The principal eigenvalues for time-periodic nonlocal dispersal operators have been studied in \cite{Bao-2019-CPAA,Hutson-2008-RMJM,Rawal-2012-JDDE,
Shen-2007-JDE,Shen-2015-JDE,Shen-2017,shen-2019,Sun-2017-JDE}. In this
paper, we adopt the approaches as in 
Berestycki, Nirenberg and Varadhan
\cite{Berestycki-1994-CPAM} and
Berestycki, Coville and Vo \cite{Berestycki-2016-JFA}
for the definition of
the generalised principal eigenvalue $\lambda_{p}(L_{\Omega}^{\tau,\mu,\sigma,m})$:
\begin{equation*}
\lambda_{p}(L_{\Omega}^{\tau,\mu,\sigma,m}):=\sup \{\lambda\in \mathbb{R}\ | \ \exists\ v\in \chi_{\Omega}^{++} \ \text{s.t.} \ (L_{\Omega}^{\tau,\mu,\sigma,m}+\lambda)[v]\leq 0 \ \text{in} \ \bar{\Omega}\times \mathbb{R} \}.
\end{equation*}

Another definition for
the generalised principal eigenvalue of $L_{\Omega}^{\tau,\mu,\sigma,m}$ is given by
\begin{equation*}
\lambda_{p}^{'}(L_{\Omega}^{\tau,\mu,\sigma,m}):=\inf \{\lambda\in \mathbb{R}\ | \ \exists\ v\in \chi_{\Omega}^{++} \ \text{s.t.} \ (L_{\Omega}^{\tau,\mu,\sigma,m}+\lambda)[v]\geq 0  \ \text{in} \ \bar{\Omega}\times \mathbb{R} \},
\end{equation*}
motivated by the works of Donsker and Varadhan in \cite{Donsker-1975-PNAS} and Berestycki, Coville and Vo in \cite{Berestycki-2016-JFA}.

Shen and Vo proved in
\cite{shen-2019} that $\lambda_{p}=\lambda_{p}^{'}$ when $\lambda_{p}$
is the principal eigenvalue and $h^{\sigma}(x)\equiv1$ for all $\sigma>0$.
Our first main result
proves that for general $h^\sigma$,
$\lambda_{p}=\lambda_{p}^{'}$ always holds,
and $\lambda_{p}$ can be characterised as the infimum of the
spectrum of $-L_{\Omega}^{\tau,\mu,\sigma,m}$, whether $\lambda_{p}$ is an eigenvalue or not. 
\begin{theorem}\label{th101}
Assume that {\rm(J)}, {\rm(H)} and {\rm(A)} hold. Then 
\begin{equation*}
\lambda_{p}(L_{\Omega}^{\tau,\mu,\sigma,m})=\lambda_{p}^{'}(L_{\Omega}^{\tau, \mu,\sigma,m})
=\lambda_{1},
\end{equation*}
where $\lambda_{1}=\inf\{Re\lambda \ | \ \lambda\in \sigma(-L_{\Omega}^{\tau,\mu,\sigma,m})\}$ and $\sigma(-L_{\Omega}^{\tau,\mu,\sigma,m})$ is the spectrum of $-L_{\Omega}^{\tau,\mu,\sigma,m}$.
\end{theorem}

Next, we turn to study the influences of the frequency $\tau$, the dispersal rate $\mu$ and the dispersal spread $\sigma$ on the generalised principal eigenvalue $\lambda_{p}(L_{\Omega}^{\tau,\mu,\sigma,m})$. The following result
establishes the monotonicity and asymptotic behaviors of the generalised principal eigenvalue $\lambda_{p}(L_{\Omega}^{\tau,\mu,\sigma,m})$ with respect to the frequency $\tau$:

\begin{theorem}\label{th12}
Assume that {\rm(J)}, {\rm(H)} and {\rm(A)}  hold. Then the following conclusions hold:
\begin{itemize}
\item[(i)] The function $\tau\mapsto\lambda_{p}(L_{\Omega}^{\tau,\mu,\sigma,m})$ is 
non-decreasing and continuous on $(0,\infty)$.
Moreover, if $\lambda_{p}(L_{\Omega}^{\tau,\mu,\sigma,m})$ is a principal eigenvalue, then the following assertions hold:
    \begin{itemize}
    \item[(a)] If $a(x,t)=\hat{a}(x)+g(t)$, then $\lambda_{p}(L_{\Omega}^{\tau,\mu,\sigma,m})$ is constant for $\tau>0$.
    \item[(b)] Otherwise $\frac{\partial\lambda_{p}}{\partial\tau}
        (L_{\Omega}^{\tau,\mu,\sigma,m})>0$ for every $\tau>0$;
    \end{itemize}

\item[(ii)] If $a\in C^{0,1}(\bar{\Omega}\times \mathbb{R})$, then there holds
\begin{equation*}
\lim\limits_{\tau\rightarrow0^{+}}\lambda_{p}(L_{\Omega}^{\tau,\mu,\sigma,m})=  \int_{0}^{1}\lambda_{p}(N_{\Omega}^{t})dt.
\end{equation*}
Here, for each fixed $t\in[0,1]$, $\lambda_{p}(N_{\Omega}^{t})$ is the generalised principal eigenvalue of the operator $N_{\Omega}^{t}$
\begin{equation*}
N_{\Omega}^{t}[v](x):=\frac{\mu}{\sigma^{m}}\bigg(\int_{\Omega}J_{\sigma}(x-y)v(y)dy
-h^{\sigma}(x)v(x)\bigg)+a(x,t)v(x);
\end{equation*}
\item[(iii)] If $a\in C^{0,1}(\bar{\Omega}\times \mathbb{R})$, then there holds
\begin{equation*}
\lim\limits_{\tau\rightarrow\infty}\lambda_{p}(L_{\Omega}^{\tau,\mu,\sigma,m})=  \lambda_{p}(N_{\Omega}),
\end{equation*}
where $\lambda_{p}(N_{\Omega})$ is the generalised principal eigenvalue of the operator $N_{\Omega}$
\begin{equation*}
N_{\Omega}[v](x):=\frac{\mu}{\sigma^{m}}\bigg(\int_{\Omega}J_{\sigma}(x-y)v(y)dy
-h^{\sigma}(x)v(x)\bigg)+\hat{a}(x)v(x)
\end{equation*}
with
\begin{equation*}
\hat{a}(x):=\int_{0}^{1} a(x,t)dt,~~~~~x\in \bar{\Omega}.
\end{equation*}
\end{itemize}
\end{theorem}

Theorem \ref{th12} is
motivated by the recent work of Liu et al.  \cite{Liu-2019-PAMS}
for time-periodic parabolic operators;
See also 
\cite{Hutson-2001-JMB, Nadin-2009-AMPA}.
Biologically,
this reflects that in a  spatio-temporal heterogeneous environment,
when the temporal variability increases,
it becomes harder for a single species to persist.
Part (i) also implies that if $a$ is a time-periodic function
with period $T$, then the generalised principal eigenvalue
is a non-increasing function of $T$.

Now, we investigate the effects of the dispersal on the generalised principal eigenvalue.
On one hand, we study the dependence of the generalised principal eigenvalue $\lambda_{p}$ on the dispersal rate $\mu$. For this purpose, we consider the non-scaled operators $L_{\Omega}^{\tau,\mu,1,0}$. On the other hand, we also
intend to understand the effects of the dispersal spread and the dispersal budget on the generalised principal eigenvalue $\lambda_{p}$. The concept of the dispersal budget was introduced by Hutson et al. \cite{Hutson-2003-JMB}. They showed that the dispersal rate is characterised
by $\frac{\mu}{\sigma^{m}}$ under proper conditions.
From the biological point of view, the species can ``choose'' to disperse a few offspring over a long distance or many offspring over a short distance or some other combinations. 

\begin{theorem}\label{th102}
Assume that {\rm(J)}, {\rm(H)} and {\rm(A)} 
hold. Then the following conclusions hold:
\begin{itemize}
\item[(i)] The function $\mu\mapsto\lambda_{p}(L_{\Omega}^{\tau,\mu,1,0})$ is continuous on $(0,\infty)$
and
there holds
\begin{equation*}
\lim\limits_{\mu\rightarrow0^{+}}\lambda_{p}(L_{\Omega}^{\tau,\mu,1,0})=  -\max_{\bar{\Omega}}\hat{a};
\end{equation*}
\item[(ii)] The function $\sigma\mapsto\lambda_{p}(L_{\Omega}^{\tau,\mu,\sigma,m})$
is continuous on $(0,\infty)$ and
\begin{itemize}
    \item[(a)] If $m>0$, then there holds
\begin{equation*}
\lim\limits_{\sigma\to\infty}\lambda_{p}(L_{\Omega}^{\tau,\mu,\sigma,m})=
-\underset{\bar{\Omega}}{\max}\ \hat{a};
\end{equation*}
\item[(b)] If $m=0$ and $\lim\limits_{\sigma\to\infty} h^{\sigma}(x)=c$, then there holds
\begin{equation*}
\lim\limits_{\sigma\to\infty}\lambda_{p}(L_{\Omega}^{\tau,\mu,\sigma,m})=
\mu c-\underset{\bar{\Omega}}{\max}\ \hat{a}.
\end{equation*}
\end{itemize}
\end{itemize}
\end{theorem}

\begin{remark}\label{re101}{\rm
For the case $m=0$, 
when the nonlocal dispersal operators take
Dirichlet boundary conditions,  $h^{\sigma}(x)\equiv1$;
When the nonlocal dispersal operators take
Neumann boundary conditions, i.e., $h^{\sigma}(x)=\int_{\Omega}J_{\sigma}(x-y)dy$, we have $\lim\limits_{\sigma\to\infty} h^{\sigma}(x)=0$.
This implies that
the boundary conditions play an important role in the persistence of the populations, i.e.  the large spread strategy with Neumann boundary conditions may be more advantageous
for species to persist, in comparison to Dirichlet boundary conditions \cite{shen-2019}.
}
\end{remark}


For later applications to time-periodic nonlinear KPP equations
with nonlocal dispersal,
we also investigate the time-periodic nonlocal dispersal operators with
Neumann boundary conditions. 
More precisely, we have
\begin{theorem}\label{th103}
Assume that {\rm(J)} and {\rm(A)} hold. If $h^{\sigma}(x)=
\int_{\Omega}J_{\sigma}(x-y)dy$, then the following conclusions hold:
\begin{itemize}
\item[(i)] There exists $\mu_{1}>0$ such that $\lambda_{p}(L_{\Omega}^{\tau,\mu,1,0})$ is the principal eigenvalue of $L_{\Omega}^{\tau,\mu,1,0}$ for all $\mu\geq\mu_{1}$. Moreover, there holds
\begin{equation*}
\lim\limits_{\mu\rightarrow\infty}\lambda_{p}(L_{\Omega}^{\tau,\mu,1,0})=
-\bar{\hat{a}},
\end{equation*}
where $\bar{\hat{a}}=\frac{1}{|\Omega|}\int_{\Omega}\hat{a}(x)dx$;
\item[(ii)] There exists $\sigma_{0}>0$ such that $\lambda_{p}(L_{\Omega}^{\tau,\mu,\sigma,m})$ is the principal eigenvalue of $L_{\Omega}^{\tau,\mu,\sigma,m}$ for all $0<\sigma\leq\sigma_{0}$. Moreover, if $J$ is symmetric with respect to each component and $0\leq m<2$, then there holds
\begin{equation*}
\lim\limits_{\sigma\to0^{+}}\lambda_{p}(L_{\Omega}^{\tau,\mu,\sigma,m})=
-\underset{\bar{\Omega}}{\max}\ \hat{a}.
\end{equation*}
\end{itemize}
\end{theorem}

For the case $m=2$, it is shown in \cite{Shen-2015-JDE} that
$\lim\limits_{\sigma\to0^{+}}\lambda_{p}(L_{\Omega}^{\tau,\mu,\sigma,2})=
\lambda_{r}$,
where $\lambda_{r}$ is the principal eigenvalue of the corresponding time-periodic random dispersal eigenvalue problem. 
For the case $m>2$, we conjecture that $\lim\limits_{\sigma\to0^{+}}\lambda_{p}(L_{\Omega}^{\tau,\mu,\sigma,m})=
-\bar{\hat{a}}$, which has been proved by \cite[Theorem 1.3 (iii)]{Su-2019-submitted} for the time-independent operators.

In the second part of this paper,
we consider the applications of previous results for the generalised
principal eigenvalue to the nonlocal dispersal equation
in spatio-temporally heterogeneous environments
\begin{equation} \label{104}
\begin{cases}
\tau u_{t}(x,t)=\frac{\mu}{\sigma^{m}}\int_{\Omega}J_{\sigma}(x-y)(u(y,t)-u(x,t))dy+
f(x,t,u(x,t)),&
~~~~~(x,t)\in \bar{\Omega}\times(0,\infty),\\
u(x,0)=u_{0}(x),&
~~~~~x\in \bar{\Omega},
\end{cases}
\end{equation}
and the time-periodic nonlocal dispersal KPP equation with Neumann boundary conditions
\begin{equation} \label{105}
\begin{cases}
\tau u_{t}(x,t)=\frac{\mu}{\sigma^{m}}\int_{\Omega}J_{\sigma}(x-y)(u(y,t)-u(x,t))dy+
f(x,t,u(x,t)),&~~~~~ (x,t)\in \bar{\Omega}\times\mathbb{R},\\
u(x,t+1)=u(x,t),&~~~~~ (x,t)\in \bar{\Omega}\times\mathbb{R},
\end{cases}
\end{equation}
where $u(x,t)$ represents the population density at location $x$ and time $t$. 
Since we only integrate over $\Omega$, we assume that diffusion takes place only in $\Omega$. The individuals may not enter or leave the domain, which is called \textbf{nonlocal Neumann boundary condition}; See \cite{Andreu-2010-AMS,Rossi-2007-JDE}.
The nonlinearity $f(x,t,u)$ satisfies the following assumptions:
\begin{itemize}
\item[($F$)] $f: \bar{\Omega}\times\mathbb{R}\times\mathbb{R}\rightarrow \mathbb{R}$ is of KPP type and satisfies:\\
(1) $f(\cdot,t,u)\in C^{1}(\bar{\Omega})$, $f(x,\cdot,u)\in C(\mathbb{R})$  and  $f(x,t,\cdot)\in C^{1}(\mathbb{R})$;\\
(2) $f(x,t,0)=0$ \ for all $(x,t)\in \bar{\Omega}\times\mathbb{R}$ and
\begin{equation*}
f(x,t+1,u)=f(x,t,u),~~~~~~~\forall (x,t,u)\in \bar{\Omega}\times\mathbb{R}\times\mathbb{R};
\end{equation*}
(3) For all \ $(x,t)\in \bar{\Omega}\times\mathbb{R}$, the function $u\mapsto f(x,t,u)/u$ is decreasing on $(0,\infty)$;\\
(4) There exists $M>0$ such that
\begin{equation*}
f(x,t,u)\leq 0,~~~~~~~\forall~ (x,t,u)\in \bar{\Omega}\times\mathbb{R}\times[M,\infty).
\end{equation*}
\end{itemize}
From now on, we set
\begin{equation*}
a(x,t)=f_{u}(x,t,0),\ \ \ \ \ (x,t)\in\bar{\Omega}\times\mathbb{R}.
\end{equation*}
Then, $L_{\Omega}^{\tau, \mu,\sigma,m}$, defined in \eqref{101}, is the linear operator associated to the linearization of \eqref{105} at $u\equiv0$.

Nonlocal dispersal evolution equations of the form \eqref{104} have
attracted a lot of attentions in recent years; See \cite{Rawal-2012-JDDE,Shen-2015-JDE,Shen-2017,shen-2019,Sun-2017-JDE}
and references therein. 
The case
$f(x,t,u)=f(x,u)$ in equations \eqref{104} has been well studied; See \cite{Bates-2007-JMAA,Berestycki-2016-JMB,Rossi-2006-JMPA,
Coville-2010-JDE,Coville-2015-DCDS,Kao-2010-DCDS,Shen-2010-JDE,Shen-2012-PAMS,
Su-2018-submitted,Su-2019-submitted,Sun-2015-DCDS,Yang-2016-DCDSA}.
We first recall the following
results of the existence and non-existence of positive time-periodic solutions to
\eqref{105} by Rawal and Shen  \cite{Rawal-2012-JDDE} and Shen and Vo  \cite{shen-2019}:
\begin{lemma}\label{th104}
Assume that {\rm(J)} and {\rm(F)} hold. Let $u(x,t;u_{0})$ be a solution of \eqref{104} with initial data $u_{0} \in C(\bar{\Omega})$, which is non-negative and
not identically zero. The following statements hold:
\begin{itemize}
\item[(i)] If $\lambda_{p}(L_{\Omega}^{\tau,\mu,\sigma,m})<0$, then \eqref{105} admits a unique solution $u^{\ast}$ in
    $\chi_{\Omega}^{++}$ and there holds
\begin{equation*}
||u(\cdot,t;u_{0})-u^{*}(\cdot,t)||_{\infty}\rightarrow 0 \ \ as \ \ t\rightarrow \infty,
\end{equation*}
~~~~~~~~where $||\cdot||_{\infty}$ is the sup norm on $C(\bar{\Omega})$;
\item[(ii)] If $\lambda_{p}(L_{\Omega}^{\tau,\mu,\sigma,m})>0$, then
\eqref{105} admits no solution in
    $\chi_{\Omega}^{+}\setminus\{0\}$ and there holds
\begin{equation*}
||u(\cdot,t;u_{0})||_{\infty}\rightarrow 0 \ \ as \ \ t\rightarrow \infty.
\end{equation*}
\end{itemize}
\end{lemma}

Now, we discuss the effects of the frequency on the persistence of populations. The following conclusion is a direct corollary of Theorem \ref{th12} and Lemma \ref{th104}.

\begin{corollary}\label{17}
Assume that {\rm(J)} and {\rm(F)} hold. Then the following statements hold:
\begin{itemize}
\item[(i)] If $\int_{0}^{1}\lambda_{p}(N_{\Omega}^{t})dt>0$, then \eqref{105} admits no solution in $\chi_{\Omega}^{+}\setminus\{0\}$ and zero solution is globally asymptotically stable for all $\tau\in(0,\infty)$;
\item[(ii)] If $\int_{0}^{1}\lambda_{p}(N_{\Omega}^{t})dt<0$, $\lambda_{p}(N_{\Omega})>0$ and $\lambda_{p}(L_{\Omega}^{\tau,\mu,\sigma,m})$ is a principal eigenvalue of the operator $L_{\Omega}^{\tau,\mu,\sigma,m}$, then there is a constant $\tau^{*}>0$ such that
    \begin{itemize}
    \item[(a)] If $\tau<\tau^{*}$, then \eqref{105} admits a unique solution $u_{\tau}^{\ast}\in \chi_{\Omega}^{++}$ that is globally asymptotically stable.
    \item[(b)] If $\tau>\tau^{*}$, then \eqref{105} admits no solution in $\chi_{\Omega}^{+}\setminus\{0\}$ and zero solution is globally asymptotically stable;
    \end{itemize}
\item[(iii)] If $\lambda_{p}(N_{\Omega})<0$, then \eqref{105} admits a unique solution $u_{\tau}^{\ast}\in \chi_{\Omega}^{++}$ that is globally asymptotically stable for all $\tau\in(0,\infty)$.
\end{itemize}
\end{corollary}

In the spatially and temporally varying environment,
Corollary \ref{17} ($ii$) suggests that increasing the 
frequency of oscillations in the resources 
may be disadvantageous to the persistence of populations.
It should be pointed out that the condition of Corollary \ref{17} ($i$)-($iii$) may be satisfied respectively; See Theorem \ref{pro301} 
for more details.

We turn to study the effects of the dispersal rate $\mu$ on the
persistence of populations. The existence and asymptotic behaviors of positive time-periodic solutions associated to \eqref{105} in the non-scaled case with $m=0$ and $\sigma=1$
are obtained as $\mu$ tends to zero or infinity.
\begin{theorem}\label{th105}
Assume that {\rm(J)} and {\rm(F)} hold. Then the following statements hold:
\begin{itemize}
\item[(i)] If $\max_{\bar{\Omega}}\hat{a}>0$, then there exists $\mu_{0}>0$ such that
\eqref{105} admits a unique solution $u_{\mu}^{\ast}\in \chi_{\Omega}^{++}$ that is globally asymptotically stable for all $\mu\in(0,\mu_{0})$. Moreover, if $\min_{\bar{\Omega}}\hat{a}>0$, then
\begin{equation*}
\underset{\mu \rightarrow 0^{+} }{\lim }\ u_{\mu}^{\ast}(x,t)=v^{\ast}(x,t) \ \ \text{uniformly in} \ \ (x,t)\in \bar{\Omega}\times\mathbb{R},
\end{equation*}
where $v^{\ast}(x,t)$ is the unique positive 1-periodic solution of the equation
$\tau v_{t}=f(x,t,v)$ for every $x\in \bar{\Omega}$.
\item[(ii)] If $\bar{\hat{a}}>0$, then there exists $\mu_{1}>0$ such that
\eqref{105} admits a unique solution $u_{\mu}^{\ast}\in \chi_{\Omega}^{++}$ that is globally asymptotically stable for all $\mu\in(\mu_{1},\infty)$.
Moreover,
\begin{equation*}
\underset{\mu \rightarrow \infty }{\lim }\ u_{\mu}^{\ast}(x,t)=v^{\ast}(t) \ \ \text{uniformly in} \ \ (x,t)\in \bar{\Omega}\times\mathbb{R},
\end{equation*}
where $v^{\ast}(t)$ is the unique positive 1-periodic solution of the  equation
\begin{equation}\label{106}
\tau v_{t}(t)=\frac{1}{|\Omega|}\int_{\Omega}f(x,t,v(t))dx.
\end{equation}

\end{itemize}
\end{theorem}

We see from Theorem \ref{th105} that the populations with small dispersal rate
can persist while the populations with large dispersal rate die out, provided that $\bar{\hat{a}}<0<\max_{\bar{\Omega}}\hat{a}$. This shows that the small dispersal rates are
better dispersal strategies than the larger ones in proper situations.

Now, we are interested in the effects of the dispersal spread and the
dispersal budget on the persistence of populations.
We establish the existence, uniqueness and stability of positive time-periodic solutions to  \eqref{105} when $\sigma $ is sufficiently small or large. Furthermore, we analyse
the asymptotic limits of the positive time-periodic solutions as $\sigma$ tends to zero or infinity. As
in \cite{Hutson-2003-JMB,shen-2019}, these asymptotics for $\sigma \ll 1$ or $\sigma \gg 1$ represent two completely different dispersal strategies:
The limit $\sigma\rightarrow 0^{+}$ can be associated to a strategy of dispersing many offspring on a short range, while the limit $\sigma \rightarrow +\infty $ corresponds to a strategy that disperses a few offspring over a long distance. More precisely, we obtain
\begin{theorem}\label{th106}
Assume that {\rm(J)} and {\rm(F)} hold. Then the following statements hold:
\begin{itemize}
\item[(i)] Let $m\geq0$. If $\max_{\bar{\Omega}}\hat{a}>0$, then there exists $\sigma_{1}>0$ such that \eqref{105} admits a unique solution $u_{\sigma}^{\ast}\in \chi_{\Omega}^{++}$ that is globally asymptotically stable for all $\sigma\in(\sigma_{1},\infty)$. Moreover, if $\min_{\bar{\Omega}}\hat{a}>0$, then
\begin{equation*}
\underset{\sigma \rightarrow \infty }{\lim }\ u_{\sigma}^{\ast}(x,t)=v^{\ast}(x,t) \ \ \text{uniformly in} \ \ (x,t)\in \bar{\Omega}\times\mathbb{R},
\end{equation*}
where $v^{\ast}(x,t)$ is the unique positive 1-periodic solution of
the equation
$\tau v_{t}=f(x,t,v)$ for every $x\in \bar{\Omega}$;
\item[(ii)] Let $0\leq m<2$. If $J$ is symmetric with respect to each component and $\max_{\bar{\Omega}}\hat{a}>0$, then there exists $\sigma_{0}>0$ such that \eqref{105} admits a unique solution $u_{\sigma}^{\ast}\in \chi_{\Omega}^{++}$ that is globally asymptotically stable for all $\sigma\in(0,\sigma_{0})$. Moreover, if $\min_{\bar{\Omega}}\hat{a}>0$, then
\begin{equation*}
\underset{\sigma \rightarrow 0^{+} }{\lim }\ u_{\sigma}^{\ast}(x,t)=v^{\ast}(x,t) \ \ \text{uniformly in} \ \ (x,t)\in \bar{\Omega}\times\mathbb{R},
\end{equation*}
where $v^{\ast}(x,t)$ is the same as in {\rm(i)}.
\end{itemize}
\end{theorem}

In addition, Shen and Xie proved in \cite{Shen-2015-JDE}
that for the case $m=2$ and $\lambda_{r}<0$, there exists $\sigma_{0}>0$ such that
\eqref{105} admits a unique solution $u_{\sigma}^{\ast}\in \chi_{\Omega}^{++}$ that is globally asymptotically stable for all $\sigma\in(0,\sigma_{0})$ and
\begin{equation*}
\underset{\sigma \rightarrow 0^{+} }{\lim }\ u_{\sigma}^{\ast}(x,t)=v(x,t) \ \ \text{uniformly in} \ \ (x,t)\in \bar{\Omega}\times\mathbb{R},
\end{equation*}
where $v$ is the positive 1-periodic solution of the corresponding reaction
diffusion equation.
For the case $m>2$, it seems reasonable to conjecture that when $\bar{\hat{a}}>0$, there exists $\sigma_{0}>0$ such that
\eqref{105} admits a unique solution $u_{\sigma}^{\ast}\in \chi_{\Omega}^{++}$ that is globally asymptotically stable for all $\sigma\in(0,\sigma_{0})$ and
there holds
\begin{equation*}
\underset{\sigma \rightarrow 0^{+} }{\lim }\ u_{\sigma}^{\ast}(x,t)=v^{*}(t) \ \ \text{uniformly in} \ \ (x,t)\in \bar{\Omega}\times\mathbb{R},
\end{equation*}
where $v^{*}(t)$ is the unique positive 1-periodic solution of \eqref{106}.
We refer interested readers to \cite[Theorem 1.8 (iii)]{Su-2019-submitted} for the time-independent case.


The rest of the paper is organised as follows. In Section 2, we first establish the equivalence of different definitions of the generalised principal eigenvalue and a characterisation of the generalised principal eigenvalue by the infimum of the spectrum. Then we study the influences of the frequency, the dispersal rate and the dispersal spread on the generalised principal eigenvalue. Section 3 is devoted to investigating the effects of the frequency, the dispersal rate and the dispersal spread on persistence criteria of populations.

\section{Time-periodic nonlocal dispersal operators}

In this section we consider the eigenvalue problem
\begin{equation} \label{200}
\begin{cases}
L_{\Omega}^{\tau,\mu,\sigma,m}[v](x,t)+\lambda v(x,t)=0,~~~~~~~& (x,t)\in \bar{\Omega}\times\mathbb{R},\\
v(x,t+1)=v(x,t),~~~~~ &(x,t)\in \bar{\Omega}\times\mathbb{R}.
\end{cases}
\end{equation}
As shown in \cite{Coville-2010-JDE,Rawal-2012-JDDE,Shen-2010-JDE,shen-2019}, the operator $L_{\Omega}^{\tau,\mu,\sigma,m}$ may not have any principal eigenvalue.
However, the generalised principal eigenvalue $\lambda_{p}(L_{\Omega}^{\tau,\mu,\sigma,m})$
can become the surrogate of the principal eigenvalue. Here,
we establish the equivalent definitions of $\lambda_{p}(L_{\Omega}^{\tau,\mu,\sigma,m})$ and
study the dependence of $\lambda_{p}(L_{\Omega}^{\tau,\mu,\sigma,m})$ on the
frequency, the dispersal rate and
the dispersal spread.

\subsection{The equivalence of the generalised principal eigenvalue}
\noindent

We consider the following general
form of nonlocal dispersal operators
\begin{equation*}
M_{\Omega}(b)[v](x,t):= -\tau v_{t}(x,t)+ \mu\int_{\Omega}J(x-y)v(y,t)dy
+b(x,t)v(x,t),\ \ \  (x,t)\in\bar{\Omega}\times\mathbb{R},
\end{equation*}
where $b\in C_{1}(\bar{\Omega}\times\mathbb{R})$. We define
\begin{equation*}
\lambda_{1}=\inf\{Re\lambda \ | \ \lambda\in \sigma(-M_{\Omega}(b))\}.
\end{equation*}

Firstly,
we recall two  lemmas in \cite[Theorem 3.3 and Proposition 6.1 (iii)]{shen-2019}.
\begin{lemma}\label{le201}
Assume that {\rm(J)} holds and $b\in C_{1}(\bar{\Omega}\times\mathbb{R})$. For any $\epsilon>0$,
there exists $b_{\epsilon}\in C_{1}(\bar{\Omega}\times\mathbb{R})$ such that the following conclusions hold:
\begin{itemize}
\item[(i)] There holds $\underset{\bar{\Omega}\times\mathbb{R}}{\max}|b-
    b_{\epsilon}|<\epsilon$;
\item[(ii)] $\lambda_{1}^{\epsilon}$ is the principal eigenvalue of $M_{\Omega}(b_{\epsilon})$,
    where $\lambda_{1}^{\epsilon}=\inf\{Re\lambda \ | \ \lambda\in \sigma(-M_{\Omega}(b_{\epsilon}))\}$;
\item[(iii)] There holds $|\lambda_{1}^{\epsilon}-\lambda_{1}|<\epsilon$.
\end{itemize}
\end{lemma}

\begin{lemma}\label{le202}
Assume that {\rm(J)} holds and $b\in C_{1}(\bar{\Omega}\times\mathbb{R})$. Then
$\lambda_{p}(M_{\Omega}(b))$ is a Lipschitz continuous function with respect to $b$. More precisely, for every $b_{1}, b_{2}\in C_{1}(\bar{\Omega}\times\mathbb{R})$, we have
\begin{equation*}
|\lambda_{p}(M_{\Omega}(b_{1}))-
\lambda_{p}(M_{\Omega}(b_{2}))|
\leq \underset{t\in[0,1]}{\sup}\
\|b_{1}(\cdot,t)-b_{2}(\cdot,t)\|_{\infty}.
\end{equation*}
\end{lemma}

Next, we prove the following two results, from which 
Theorem \ref{th101} follows as a consequence.
\begin{theorem}\label{th201}
Assume that {\rm(J)} holds and $b\in C_{1}(\bar{\Omega}\times\mathbb{R})$. Then there holds
\begin{equation*}
\lambda_{p}(M_{\Omega}(b))=\lambda_{1}.
\end{equation*}
\end{theorem}

\begin{proof}
The proof is divided into two cases.

\textbf{Case 1.} We prove the result under the additional assumption that $\lambda_{1}$ is the principal eigenvalue. By the
definition of the principal eigenvalue, there is $\varphi_{1}\in \chi_{\Omega}^{++}$ such that
\begin{equation*}
M_{\Omega}(b)[\varphi_{1}]+\lambda_{1}\varphi_{1}=0 ~~~\ \ \ \ \text{in}~~ \ \bar{\Omega}\times\mathbb{R}.
\end{equation*}
Thanks to the definition of $\lambda_{p}(M_{\Omega}(b))$, we have $\lambda_{1}\leq \lambda_{p}(M_{\Omega}(b))$. It remains to establish the inequality $\lambda_{p}(M_{\Omega}(b))\leq \lambda_{1}$, which is similar to the proof of \cite[Theorem 2.3]{shen-2019}. Here, we omit it. Thus, we get $\lambda_{p}(M_{\Omega}(b))=\lambda_{1}$.

\textbf{Case 2.} If $\lambda_{1}$ is not the principal eigenvalue, we can use an approximation argument. More precisely,
applying Lemma \ref{le201}, we find that for each $\epsilon>0$, there exists $b_{\epsilon}\in C_{1}(\bar{\Omega}\times\mathbb{R})$ such that
\begin{equation}\label{207}
\underset{\bar{\Omega}\times\mathbb{R}}{\max}\ |b_{\epsilon}-b|< \epsilon, \
|\lambda_{1}-\lambda_{1}^{\epsilon}|<\epsilon
\end{equation}
and $\lambda_{1}^{\epsilon}$ is the principal eigenvalue of
$M_{\Omega}(b_{\epsilon})$. Then, we apply \textbf{Case 1} to conclude
\begin{equation}\label{208}
\lambda_{p}(M_{\Omega}(b_{\epsilon}))=\lambda_{1}^{\epsilon}.
\end{equation}
Since $\lambda_{p}(M_{\Omega}(b))$ is Lipschitz continuous with respect to $b$ in Lemma \ref{le202} and the inequalities (\ref{207}), setting $\epsilon\rightarrow0$ in
(\ref{208}) yields
$
\lambda_{p}(M_{\Omega}(b))=\lambda_{1}$.
\end{proof}

\begin{theorem}\label{th202}
Assume that {\rm(J)} holds and $b\in C_{1}(\bar{\Omega}\times\mathbb{R})$. Then there holds
\begin{equation*}
\lambda_{p}(M_{\Omega}(b))=\lambda_{p}^{'}(M_{\Omega}(b)).
\end{equation*}
\end{theorem}

\begin{proof}
We first show that $\lambda_{p}(M_{\Omega}(b))\leq \lambda_{p}^{'}(M_{\Omega}(b))$. Let us assume by contradiction that
\begin{equation*}
\lambda_{p}^{'}(M_{\Omega}(b))<\lambda_{p}(M_{\Omega}(b)).
\end{equation*}
Pick now $\lambda\in (\lambda_{p}^{'}(M_{\Omega}(b)),\lambda_{p}(M_{\Omega}(b)))$, then, by the definition of $\lambda_{p}(M_{\Omega}(b))$ and $\lambda_{p}^{'}(M_{\Omega}(b))$, there exist $\varphi\in \chi^{++}_{\Omega}$ and $\psi\in \chi^{++}_{\Omega}$ such that
\begin{align}\label{201}
M_{\Omega}(b)[\varphi](x,t)+\lambda \varphi(x,t)\leq 0 &~~~\ \ \ \ \text{in}~~\  \bar{\Omega}\times\mathbb{R},\\
M_{\Omega}(b)[\psi](x,t)+\lambda \psi(x,t)\geq 0 &~~~\ \ \ \ \text{in}~~ \ \bar{\Omega}\times\mathbb{R}.
\end{align}
By taking $\lambda$ bigger if necessary, we assume that $\psi$ satisfies
\begin{equation}\label{203}
M_{\Omega}(b)[\psi](x,t)+\lambda \psi(x,t)> 0 ~~~\ \ \ \ \text{in}~~ \ \bar{\Omega}\times\mathbb{R}.
\end{equation}

Set $w:=\frac{\psi}{\varphi}\in \chi^{++}_{\Omega}$. Using (\ref{201}), a direct computation yields
\begin{align*}
M_{\Omega}(b)[\psi]=&M_{\Omega}(b)[w\varphi]\\
=&-\tau(w\varphi)_{t}+\mu\int_{\Omega}J(x-y)w(y,t)\varphi(y,t)dy
+b(x,t)w\varphi\\
=&-\tau w_{t}\varphi+\mu\int_{\Omega}J(x-y)\varphi(y,t)(w(y,t)-w(x,t))dy-\lambda w\varphi\\
&~~+w(-\tau \varphi_{t}+\mu\int_{\Omega}J(x-y)\varphi(y,t)dy+b(x,t)\varphi
+\lambda \varphi)\\
\leq&-\tau w_{t}\varphi+\mu\int_{\Omega}J(x-y)\varphi(y,t)(w(y,t)-w(x,t))dy-\lambda\psi.
\end{align*}
By (\ref{203}), we find
\begin{equation}\label{204}
0<-\tau w_{t}\varphi+\mu\int_{\Omega}J(x-y)\varphi(y,t)(w(y,t)-w(x,t))dy
~~~\ \ \ \ \text{in}~~\ \bar{\Omega}\times\mathbb{R}.
\end{equation}
Since $w\in \chi^{++}_{\Omega}$, there exists $(x_{0},t_{0})\in \bar{\Omega}\times[0,1]$ such that
\begin{equation*}
w(x_{0},t_{0})=\underset{\bar{\Omega}\times\mathbb{R}}{\max}\ w , \ \ w_{t}(x_{0},t_{0})=0.
\end{equation*}
Hence, setting $(x,t)=(x_{0},t_{0})$ in (\ref{204}) yields
\begin{equation*}
0<-\tau w_{t}(x_{0},t_{0})\varphi(x_{0},t_{0})+\mu\int_{\Omega}J(x_{0}-y)\varphi(y,t_{0})
(w(y,t_{0})-w(x_{0},t_{0}))dy\leq 0,
\end{equation*}
which is a contradiction. Therefore, $\lambda_{p}(M_{\Omega}(b))\leq \lambda_{p}^{'}(M_{\Omega}(b))$.

To complete the proof, it suffices to establish 
\begin{equation}\label{eq:temp-11}
\lambda_{p}^{'}(M_{\Omega}(b))\leq\lambda_{p}(M_{\Omega}(b))+2\delta \ \ \ \
\text{for all}\ \ \delta>0.
\end{equation}

We claim that for any $\delta>0$, there exists $\varphi_{\delta}\in \chi^{++}_{\Omega}$ such that
\begin{equation*}
M_{\Omega}(b)[\varphi_{\delta}]+(\lambda_{p}(M_{\Omega}(b))+2\delta) \varphi_{\delta}\geq 0 ~~~\ \ \ \ \text{in}~~ \ \bar{\Omega}\times\mathbb{R}.
\end{equation*}
Indeed, thanks to Lemma \ref{le201} and Theorem \ref{th201}, there is $b_{\delta}\in C_{1}(\bar{\Omega}\times\mathbb{R})$ such that
\begin{equation}\label{205}
\underset{\bar{\Omega}\times\mathbb{R}}{\max}\ |b_{\delta}-b|< \delta, \  |\lambda_{p}(M_{\Omega}(b))-\lambda_{p}(M_{\Omega}(b_{\delta}))|< \delta
\end{equation}
and $\lambda_{p}(M_{\Omega}(b_{\delta}))$ is the principal eigenvalue of
$M_{\Omega}(b_{\delta})$.
Thus, there exists $\varphi_{\delta}\in \chi^{++}_{\Omega}$ such that
\begin{equation}\label{206}
M_{\Omega}(b_{\delta})[\varphi_{\delta}]+\lambda_{p}(M_{\Omega}(b_{\delta})) \varphi_{\delta}=0 ~~~\ \ \ \ \text{in}~~ \ \bar{\Omega}\times\mathbb{R}.
\end{equation}
Owing to (\ref{205}) and (\ref{206}), we get
\begin{align*}
&M_{\Omega}(b)[\varphi_{\delta}]+(\lambda_{p}(M_{\Omega}(b))+2\delta) \varphi_{\delta}\\
=&M_{\Omega}(b_{\delta})[\varphi_{\delta}]+(b(x,t)-b_{\delta}(x,t))\varphi_{\delta}
+(\lambda_{p}(M_{\Omega}(b))+2\delta) \varphi_{\delta}\\
=&(\lambda_{p}(M_{\Omega}(b))-\lambda_{p}(M_{\Omega}(b_{\delta})))\varphi_{\delta}+
(b(x,t)-b_{\delta}(x,t))\varphi_{\delta}+2\delta\varphi_{\delta}\\
\geq& -\delta\varphi_{\delta}-\delta\varphi_{\delta}+2\delta\varphi_{\delta}=0
\end{align*}
The proof of the claim is complete.

Moreover, it follows from this claim and the definition of $\lambda_{p}^{'}(M_{\Omega}(b))$ that \eqref{eq:temp-11} holds.
In conclusion, we obtain $\lambda_{p}(M_{\Omega}(b))=\lambda_{p}^{'}(M_{\Omega}(b))$.
\end{proof}

\subsection{Influences of the frequency}
\noindent

This subsection concerns the dependence of  $\lambda_{p}(L_{\Omega}^{\tau,\mu,\sigma,m})$ on $\tau$. Consider
\begin{equation*}
M_{\tau}[v](x,t):= -\tau v_{t}(x,t)+ \mu\int_{\Omega}J(x-y)v(y,t)dy
+b(x,t)v(x,t),\ \ \  (x,t)\in\bar{\Omega}\times\mathbb{R},
\end{equation*}
where $b\in C_{1}(\bar{\Omega}\times\mathbb{R})$.
Our goal is to prove the following two results, from which  the conclusions of
Theorem \ref{th12} follow:

\begin{theorem}\label{th401}
Assume that {\rm(J)} holds and $b\in C_{1}(\bar{\Omega}\times\mathbb{R})$. Then the function $\tau\mapsto\lambda_{p}(M_{\tau})$ is continuous non-decreasing on $(0,\infty)$. Moreover, if $\lambda_{p}(M_{\tau})$ is a principal eigenvalue, then the following assertions hold:
    \begin{itemize}
    \item[(i)] If $b(x,t)=\hat{b}(x)+g(t)$, then $\lambda_{p}(M_{\tau})$ is constant for $\tau>0$;
    \item[(ii)] Otherwise $\frac{\partial\lambda_{p}}{\partial\tau}
        (M_{\tau})>0$ for every $\tau>0$.
    \end{itemize}
\end{theorem}

\begin{proof}
The proof is divided into two cases.

\textbf{Case 1.} We prove the result under the additional assumption that $\lambda_{p}(M_{\tau})$ is a principal eigenvalue for all $\tau>0$. By the
definition of the principal eigenvalue, there exists $\varphi_{\tau}\in \chi_{\Omega}^{++}$ s.t. 
\begin{equation}\label{401}
\begin{cases}
M_{\tau}[\varphi_{\tau}](x,t)=-\tau \partial_{t}\varphi_{\tau}+\mu \int_{\Omega}J(x-y)\varphi_{\tau}(y,t)dy+b(x,t)\varphi_{\tau}=
-\lambda_{p}(M_{\tau})\varphi_{\tau}& \text{ in }\ \bar{\Omega}\times[0,1],\\
\varphi_{\tau}(x,1)=\varphi_{\tau}(x,0)&
\text{ in }\ \bar{\Omega}.
\end{cases}
\end{equation}
Note that there is $\psi_{\tau}\in \chi_{\Omega}^{++}$ such that $\psi_{\tau}$ satisfies the adjoint problem of \eqref{401}
\begin{equation}\label{402}
\begin{cases}
M_{\tau}^{*}[\psi_{\tau}](x,t):=\tau \partial_{t}\psi_{\tau}+\mu \int_{\Omega}J(x-y)\psi_{\tau}(y,t)dy+b(x,t)\psi_{\tau}=
-\lambda_{p}(M_{\tau})\psi_{\tau}& \text{ in }\ \bar{\Omega}\times[0,1],\\
\psi_{\tau}(x,1)=\psi_{\tau}(x,0)&
\text{ in }\ \bar{\Omega}.
\end{cases}
\end{equation}
For convenience, we denote $\mathcal{C}:=\Omega\times(0,1)$. We normalize $\varphi_{\tau}$ and $\psi_{\tau}$ such that $\int_{\mathcal{C}}\varphi_{\tau}^{2}=
\int_{\mathcal{C}}\varphi_{\tau}\psi_{\tau}=1$ for any $\tau>0$.

A family of closed operators $\{M_{\tau}\}_{\tau>0}$ is a holomorphic family by \cite[Charpter 7, Section 2.1]{Kato-1995}.
As $\lambda_{p}(M_{\tau})$ is an isolated eigenvalue, the continuous differentiability of $\tau \mapsto \big(\lambda_{p}(M_{\tau}),\varphi_{\tau}\big)$ follows from
the classical perturbation theory in \cite[Charpter 7, Section 6.2]{Kato-1995}.
We can differentiate the equation \eqref{401} with respect to $\tau$ to find
\begin{equation*}
\begin{cases}
-\partial_{t}\varphi_{\tau}+M_{\tau}[\varphi_{\tau}']=
-\frac{\partial \lambda_{p}(M_{\tau})}{\partial \tau}\varphi_{\tau}-\lambda_{p}(M_{\tau})\varphi_{\tau}' &~~~\text{ in }\ \bar{\Omega}\times[0,1],\\
\varphi_{\tau}'(x,1)=\varphi_{\tau}'(x,0)&
~~~\text{ in }\ \bar{\Omega}.
\end{cases}
\end{equation*}
Multiplying the above equation by $\psi_{\tau}$ and integrating the resulting equation over $\mathcal{C}$, we obtain
\begin{equation*}
-\int_{\mathcal{C}}\psi_{\tau}\partial_{t}\varphi_{\tau}+\int_{\mathcal{C}}
M_{\tau}[\varphi_{\tau}']\psi_{\tau}=-\frac{\partial \lambda_{p}(M_{\tau})}
{\partial \tau}\int_{\mathcal{C}}\varphi_{\tau}\psi_{\tau} -\lambda_{p}(M_{\tau})\int_{\mathcal{C}}\varphi_{\tau}'\psi_{\tau}.
\end{equation*}
By the adjoint problem \eqref{402} and the normalization $\int_{\mathcal{C}}\varphi_{\tau}\psi_{\tau}=1$, we find that
\begin{equation*}
\frac{\partial\lambda_{p}(M_{\tau})}{\partial\tau}=
\int_{\mathcal{C}}\psi_{\tau}\partial_{t}\varphi_{\tau}.
\end{equation*}
Due to the definition of $M_{\tau}$ and $M_{\tau}^{*}$, we derive
\begin{align*}
\int_{\mathcal{C}}\psi_{\tau}\partial_{t}\varphi_{\tau}&=\frac{1}{2\tau}
\int_{\mathcal{C}}\psi_{\tau}\big(M_{\tau}^{*}[\varphi_{\tau}]-
M_{\tau}[\varphi_{\tau}]\big)\\
&=\frac{1}{2\tau}
\int_{\mathcal{C}}\big(\varphi_{\tau}M_{\tau}[\psi_{\tau}]-
\psi_{\tau}M_{\tau}[\varphi_{\tau}]\big)\\
&=\frac{1}{2\tau}
\bigg(K_{\tau}(\psi_{\tau})-K_{\tau}(\varphi_{\tau})\bigg),
\end{align*}
where functional $K_{\tau}$ is defined by
\begin{equation*}
K_{\tau}(\zeta):=\int_{\mathcal{C}}\varphi_{\tau}\psi_{\tau}
\bigg(\frac{M_{\tau}[\zeta]}{\zeta}\bigg),\ \ \ \ \zeta\in\chi_{\Omega}^{++}.
\end{equation*}
We claim that
\begin{claim}\label{cl201}
For any $\zeta\in\chi_{\Omega}^{++}$, we have
\begin{equation*}
K_{\tau}(\zeta)-K_{\tau}(\varphi_{\tau})\geq0.
\end{equation*}
\end{claim}
Assume for the moment that the claim holds true, then
it implies that
\begin{equation}\label{406}
\frac{\partial\lambda_{p}(M_{\tau})}{\partial\tau}=\frac{1}{2\tau}
\bigg(K_{\tau}(\psi_{\tau})-K_{\tau}(\varphi_{\tau})\bigg)\geq0
\ \ \ \ \text{for all } \tau>0.
\end{equation}

It remains to prove parts ($i$) and ($ii$). When $b(x,t)=\hat{b}(x)+g(t)$ for some 1-periodic function $g(t)$,
we set $\phi_{\tau}(x,t):=e^{-\frac{1}{\tau}\int_{0}^{t}g(s)ds}\varphi_{\tau}(x,t)$,
which satisfies
\begin{equation*}
\begin{cases}
-\tau \partial_{t}\phi_{\tau}+\mu \int_{\Omega}J(x-y)\phi_{\tau}(y,t)dy+\hat{b}(x)\phi_{\tau}=
-\lambda_{p}(M_{\tau})\varphi_{\tau}&\ \ \ \text{ in }\ \bar{\Omega}\times[0,1],\\
\phi_{\tau}(x,1)=\phi_{\tau}(x,0)&\ \ \ \text{ in }\ \bar{\Omega}.
\end{cases}
\end{equation*}
It is clear that $\lambda_{p}(M_{\tau})$ is constant for $\tau>0$. This
proves part ($i$).

Finally, we show that $\frac{\partial\lambda_{p}(M_{\tau})}{\partial\tau}>0$ for every $\tau>0$ if $b(x,t)$ does not take the form of $b(x,t)=\hat{b}(x)+g(t)$.
Suppose that there is some $\tau_{0}>0$ such that
$\frac{\partial\lambda_{p}(M_{\tau_{0}})}{\partial\tau}=0$. According to
the formula \eqref{407} and $J(0)>0$, we obtain
\begin{equation*}
\frac{\varphi_{\tau_{0}}(x,t)}{\psi_{\tau_{0}}(x,t)}\cdot
\frac{\psi_{\tau_{0}}(y,t)}{\varphi_{\tau_{0}}(y,t)}\equiv1
\ \ \ \text{for each } x,y\in\bar{\Omega},\ t\in[0,1].
\end{equation*}
Thus, we have $\varphi_{\tau_{0}}=c(t)\psi_{\tau_{0}}$ for some 1-periodic function $c(t)>0$. Substituting $\varphi_{\tau_{0}}=c(t)\psi_{\tau_{0}}$ into $M_{\tau_{0}}[\varphi_{\tau_{0}}]=-\lambda_{p}(M_{\tau_{0}})
\varphi_{\tau_{0}}$ and using $M_{\tau_{0}}^{*}[\psi_{\tau_{0}}]=-\lambda_{p}(M_{\tau_{0}})
\psi_{\tau_{0}}$, we deduce that
\begin{equation*}
c'(t)\psi_{\tau_{0}}+2c(t)\partial_{t}\psi_{\tau_{0}}=0.
\end{equation*}
It then follows that $\partial_{t} ln\psi_{\tau_{0}}=-\frac{c'(t)}{2c(t)}$ in $\mathcal{C}$, which depends only on $t$. Hence, $\psi_{\tau_{0}}$ is of the form
$\psi_{\tau_{0}}=X_{\tau_{0}}(x)T_{\tau_{0}}(t)$ with some 1-periodic
function $T_{\tau_{0}}(t)>0$ in $[0,1]$ and function $X_{\tau_{0}}(x)>0$
in $\bar{\Omega}$. By $M_{\tau_{0}}^{*}[\psi_{\tau_{0}}]=-\lambda_{p}(M_{\tau_{0}})$, we have
\begin{equation*}
\tau_{0}\frac{T_{\tau_{0}}'(t)}{T_{\tau_{0}}(t)}+
\frac{\mu\int_{\Omega}J(x-y)X_{\tau_{0}}(y)dy}{X_{\tau_{0}}(x)}+b(x,t)
=-\lambda_{p}(M_{\tau_{0}}).
\end{equation*}
Thus, it is necessary that $b$ has the form of $b(x,t)=\hat{b}(x)+g(t)$,
which contradicts the previous assumption. This completes the proof of part ($ii$).

\textbf{Case 2.} If $\lambda_{p}(M_{\tau})$ is not the principal eigenvalue for some $\tau>0$, then we can use an approximation argument.
More precisely, applying Lemma \ref{le201} and Theorem \ref{th201}, we find that for each $\epsilon>0$, there exists $b_{\epsilon}\in C_{1}(\bar{\Omega}\times\mathbb{R})$ such that
\begin{equation*}
\underset{\bar{\Omega}\times\mathbb{R}}{\max}\ |b_{\epsilon}-b|< \epsilon, \
|\lambda_{p}(M_{\tau}(b_{\epsilon}))-\lambda_{p}(M_{\tau}(b_{\epsilon}))|<\epsilon
\end{equation*}
and $\lambda_{p}(M_{\tau}(b_{\epsilon}))$ is the principal eigenvalue for all $\tau>0$, where $M_{\tau}(b_{\epsilon})$
is $M_{\tau}(b)$ with $b$ being replaced by $b_{\epsilon}$.
We then apply \textbf{Case 1} to conclude that for each $\epsilon>0$,
the function $\tau\mapsto\lambda_{p}(M_{\tau}(b_{\epsilon}))$ is continuous non-decreasing on $(0,\infty)$, i.e., for every $\tau_{0}>0$,
there exists $\delta_{0}>0$ such that for all $|\tau-\tau_{0}|<\delta
_{0}$, we have
\begin{equation*}
\big|\lambda _{p}(M_{\tau}(b_{\epsilon}))-\lambda _{p}(M_{\tau_{0}}(b_{\epsilon}))\big|<\epsilon.
\end{equation*}
By Lemma \ref{le202}, $\lambda _{p}(M_{\tau}(b))$
is Lipschitz continuous with respect to $b$, i.e.,%
\begin{equation*}
\big|\lambda _{p}(M_{\tau}(b))-\lambda _{p}(M_{\tau}(b_{\epsilon}))\big|\leq \underset{t\in[0,1]}{\sup}\
\|b(\cdot,t)-b_{\epsilon }(\cdot,t)\|_{\infty}<
\epsilon.
\end{equation*}

Hence, for every given constant $\epsilon >0$, there exist $\delta
_{0}>0 $ and $b_{\epsilon }\in C_{1}(\bar{\Omega }\times \mathbb{R})$ such that for all $%
|\tau -\tau_{0}|<\delta _{0}$, we have%
\begin{align*}
&|\lambda _{p}(M_{\tau_{0}}(b))-\lambda _{p}(M_{\tau}(b))| \\
\leq &|\lambda _{p}(M_{\tau_{0}}(b))-\lambda _{p}(M_{\tau_{0}}(b_{\epsilon }))|+|\lambda _{p}(M_{\tau_{0}}(b_{\epsilon }))-\lambda _{p}(M_{\tau}(b_{\epsilon }))| 
+|\lambda _{p}(M_{\tau}(b))-\lambda _{p}(M_{\tau}(b_{\epsilon }))| \\
<&\epsilon+\epsilon+\epsilon=3\epsilon,
\end{align*}%
which implies that $\lambda _{p}(M_{\tau}(b))$ is continuous with respect to $\tau$. Thus, the function $\tau\mapsto\lambda_{p}(M_{\tau}(b))$ is continuous non-decreasing on $(0,\infty)$. The proof  is complete.
\end{proof}

\begin{proof}[Proof of  Claim {\rm\ref{cl201}}]
First, we claim that $\varphi_{\tau}$ is a critical point of $K_{\tau}$ in the sense that
\begin{equation}\label{400}
\mathbf{D} K_{\tau}(\varphi_{\tau})\eta=0 \ \ \ \ \text{for all } \eta\in\chi_{\Omega},
\end{equation}
where $\mathbf{D}K_{\tau}(\varphi_{\tau})$ is the Fr\'{e}chet derivative of $K_{\tau}$ at the point $\varphi_{\tau}\in\chi_{\Omega}^{++}$.

For any $\eta\in \chi_{\Omega}$, we have
\begin{equation*}
\mathbf{D} K_{\tau}(\varphi_{\tau})\eta=
\int_{\mathcal{C}}\varphi_{\tau}\psi_{\tau}
\bigg(\frac{M_{\tau}[\eta]}{\varphi_{\tau}}-
\frac{M_{\tau}[\varphi_{\tau}]\eta}{\varphi_{\tau}^{2}}\bigg).
\end{equation*}

On  one hand, it follows from $M_{\tau}[\varphi_{\tau}]=-\lambda_{p}\varphi_{\tau}$ and $M_{\tau}^{*}[\psi_{\tau}]=-\lambda_{p}\psi_{\tau}$ that
\begin{equation}\label{405}
\begin{split}
\mathbf{D} K_{\tau}(\varphi_{\tau})\eta&=
\int_{\mathcal{C}}\bigg(\psi_{\tau}M_{\tau}[\eta]-
\frac{M_{\tau}[\varphi_{\tau}]\psi_{\tau}\eta}{\varphi_{\tau}}\bigg)\\
&=\int_{\mathcal{C}}\bigg(M_{\tau}^{*}[\psi_{\tau}]\eta-
\frac{M_{\tau}[\varphi_{\tau}]\psi_{\tau}\eta}{\varphi_{\tau}}\bigg)\\
&=0.
\end{split}
\end{equation}
On the other hand, a simple calculation yields
\begin{equation}\label{403}
\begin{split}
&\mathbf{D} K_{\tau}(\varphi_{\tau})\eta\\
=&
\int_{\mathcal{C}}\varphi_{\tau}\psi_{\tau}
\bigg(\frac{M_{\tau}[\eta]}{\varphi_{\tau}}-
\frac{M_{\tau}[\varphi_{\tau}]\eta}{\varphi_{\tau}^{2}}\bigg)\\
=&
\int_{\mathcal{C}}\varphi_{\tau}\psi_{\tau}
\frac{-\tau\big(\eta_{t}\varphi_{\tau}-(\varphi_{\tau})_{t}\eta\big)+
\mu\int_{\Omega}J(x-y)\big[\eta(y,t)\varphi_{\tau}(x,t)-
\varphi_{\tau}(y,t)\eta(x,t)\big]dy}
{\varphi_{\tau}^{2}}.
\end{split}
\end{equation}

A direct calculation shows that
\begin{align*}
&K_{\tau}(\zeta)-K_{\tau}(\varphi_{\tau})\\
=&\int_{\mathcal{C}}\varphi_{\tau}
\psi_{\tau}\bigg(\frac{M_{\tau}[\zeta]}{\zeta}-\frac{M_{\tau}[\varphi_{\tau}]}
{\varphi_{\tau}}\bigg)\\
=&\int_{\mathcal{C}}\varphi_{\tau}\psi_{\tau}
\bigg(\frac{-\tau\zeta_{t}}{\zeta}+\frac{\tau(\varphi_{\tau})_{t}}{\varphi_{\tau}}\bigg)+
\mu\int_{\mathcal{C}}\varphi_{\tau}\psi_{\tau}\int_{\Omega}J(x-y)\bigg(
\frac{\zeta(y,t)}{\zeta(x,t)}-
\frac{\varphi_{\tau}(y,t)}{\varphi_{\tau}(x,t)}\bigg)dy.
\end{align*}
Taking $\eta=\varphi_{\tau} ln\Big(\frac{\zeta}{\varphi_{\tau}}\Big)$ in \eqref{403}, we obtain
\begin{equation}\label{404}
\begin{split}
&\mathbf{D} K_{\tau}(\varphi_{\tau})\eta\\
=&\int_{\mathcal{C}}\varphi_{\tau}\psi_{\tau}
\bigg(\frac{-\tau\zeta_{t}}{\zeta}+\frac{\tau(\varphi_{\tau})_{t}}
{\varphi_{\tau}}\bigg)+\mu\int_{\mathcal{C}}\varphi_{\tau}\psi_{\tau}
\int_{\Omega}J(x-y)\frac{\varphi_{\tau}(y,t)}{\varphi_{\tau}(x,t)} ln\bigg(\frac{\zeta(y,t)\varphi_{\tau}(x,t)}{\zeta(x,t)\varphi_{\tau}(y,t)}
\bigg)dy.
\end{split}
\end{equation}
By formulas \eqref{405} and \eqref{404}, we have
\begin{equation}\label{407}
\begin{split}
&K_{\tau}(\zeta)-K_{\tau}(\varphi_{\tau})\\
=&\mathbf{D} K_{\tau}(\varphi_{\tau})\eta+\mu\int_{\mathcal{C}}\varphi_{\tau}\psi_{\tau}
\int_{\Omega}J(x-y)\frac{\varphi_{\tau}(y,t)}{\varphi_{\tau}(x,t)} \bigg[\frac{\zeta(y,t)\varphi_{\tau}(x,t)}{\zeta(x,t)\varphi_{\tau}(y,t)}-1-
ln\bigg(\frac{\zeta(y,t)\varphi_{\tau}(x,t)}{\zeta(x,t)\varphi_{\tau}(y,t)}
\bigg)\bigg]dy\\
=&\mu\int_{\mathcal{C}}\varphi_{\tau}\psi_{\tau}
\int_{\Omega}J(x-y)\frac{\varphi_{\tau}(y,t)}{\varphi_{\tau}(x,t)} \bigg[\frac{\zeta(y,t)\varphi_{\tau}(x,t)}{\zeta(x,t)\varphi_{\tau}(y,t)}-1-
ln\bigg(\frac{\zeta(y,t)\varphi_{\tau}(x,t)}{\zeta(x,t)\varphi_{\tau}(y,t)}
\bigg)\bigg]dy.
\end{split}
\end{equation}

Define
\begin{equation*}
f(z)=z-1-lnz,\ \ \ z\in(0,\infty).
\end{equation*}
As $f(z)\geq0$ 
and $f(z)=0$ if and only if $z=1$, thus we obtain
$K_{\tau}(\zeta)-K_{\tau}(\varphi_{\tau})\geq0$.
\end{proof}

\begin{theorem}\label{th402}
Assume that {\rm(J)} holds and $b\in C_{1}^{0,1}(\bar{\Omega}\times\mathbb{R})$. Then the followings 
hold:
\begin{itemize}
\item[(i)] There holds
\begin{equation*}
\lim\limits_{\tau\rightarrow0^{+}}\lambda_{p}(M_{\tau})=  \int_{0}^{1}\lambda_{p}(N_{\Omega}^{t})dt.
\end{equation*}
Here, for each fixed $t\in[0,1]$, $\lambda_{p}(N_{\Omega}^{t})$ is the generalised principal eigenvalue of the operator $N_{\Omega}^{t}$
\begin{equation*}
N_{\Omega}^{t}[v](x):=\mu\int_{\Omega}J_{\sigma}(x-y)v(y)dy
+b(x,t)v(x);
\end{equation*}
\item[(ii)] There holds
\begin{equation*}
\lim\limits_{\tau\rightarrow\infty}\lambda_{p}(M_{\tau})=  \lambda_{p}(N_{\Omega}),
\end{equation*}
where $\lambda_{p}(N_{\Omega})$ is the generalised principal eigenvalue of the operator $N_{\Omega}$
\begin{equation*}
N_{\Omega}[v](x):=\mu\int_{\Omega}J_{\sigma}(x-y)v(y)dy
+\hat{b}(x)v(x).
\end{equation*}
\end{itemize}
\end{theorem}

\begin{proof}
(i) The proof is divided into two cases.

\textbf{Case 1.} We prove the result under the additional assumption that $\lambda_{p}(N_{\Omega}^{t})$ is a principal eigenvalue for all $t\in[0,1]$. For fixed $t\in[0,1]$, there is $v(\cdot,t)\in C(\bar{\Omega})$ and $v(\cdot,t)>0$ in $\bar{\Omega}$ s.t. 
\begin{equation}\label{408}
N^{t}_{\Omega}[v](x,t)+\lambda_{p}(N_{\Omega}^{t})v(x,t)=0\ \ \ \text{in }\ \bar{\Omega}.
\end{equation}
It follows from the 
perturbation theory \cite[Charpter 7, Section 6.2]{Kato-1995} that $v\in C^{1}([0,1];C(\bar{\Omega}))$ and $v(x,t+1)=v(x,t)$.

Define $\varphi(x,t)=\rho(t)v(x,t)$ for 1-periodic function
\begin{equation*}
\rho(t)=e^{\frac{1}{\tau}\big[t\int_{0}^{1}\lambda_{p}(N_{\Omega}^{s})ds
-\int_{0}^{t}\lambda_{p}(N_{\Omega}^{s})ds\big]}.
\end{equation*}
Given arbitrary $\epsilon_{0}>0$, there is sufficiently small $\tau_{0}>0$ such that $\tau|\partial_{t}v|\leq\epsilon_{0} v$ for all $\tau\leq\tau_{0}$. Moreover, a direct calculation yields
\begin{equation*}
M_{\tau}[\varphi]+\bigg(\int_{0}^{1}\lambda_{p}
(N_{\Omega}^{t})dt-\epsilon_{0}\bigg)\varphi
\leq-\tau\varphi_{t}+N_{\Omega}^{t}[\varphi]+\bigg(\int_{0}^{1}\lambda_{p}
(N_{\Omega}^{t})dt-\epsilon_{0}\bigg)\varphi
\leq 0
\end{equation*}
By the definition of $\lambda_{p}(M_{\tau})$, we know that
\begin{equation}\label{409}
\int_{0}^{1}\lambda_{p}(N_{\Omega}^{t})dt-\epsilon_{0}
\leq\lambda_{p}(M_{\tau})\ \ \ \ \text{for all } \tau\leq\tau_{0}.
\end{equation}

In a similar manner, we obtain
\begin{equation*}
M_{\tau}[\varphi]+\bigg(\int_{0}^{1}\lambda_{p}
(N_{\Omega}^{t})dt+\epsilon_{0}\bigg)\varphi
\geq-\tau\varphi_{t}+N_{\Omega}^{t}[\varphi]+\bigg(\int_{0}^{1}\lambda_{p}
(N_{\Omega}^{t})dt+\epsilon_{0}\bigg)\varphi
\geq0.
\end{equation*}
By the definition of $\lambda_{p}'(M_{\tau})$, we know that
\begin{equation}\label{410}
\lambda_{p}'(M_{\tau})\leq\int_{0}^{1}\lambda_{p}(N_{\Omega}^{t})dt
+\epsilon_{0}\ \ \ \ \text{for all } \tau\leq\tau_{0}.
\end{equation}

Combining Theorem \ref{th202} and  inequalities \eqref{409}, \eqref{410}, we obtain
\begin{equation*}
\lim\limits_{\tau\rightarrow0^{+}}\lambda_{p}(M_{\tau})=  \int_{0}^{1}\lambda_{p}(N_{\Omega}^{t})dt.
\end{equation*}

\textbf{Case 2.} If $\lambda_{p}(N_{\Omega}^{t})$ is not a principal eigenvalue for some $t\in[0,1]$, then we can use similar approximation argument as in \textbf{Case 2} in the proof of Theorem \ref{th401}
to deduce the result.

\medskip
(ii) The proof is also divided into two cases.

\textbf{Case 1.} We prove the result under the additional assumption that $\lambda_{p}(M_{\tau})$ is a principal eigenvalue for all $\tau>0$.
Choose a sequence of $\{\tau_{n}\}_{n=1}^{\infty}$ such that $\tau_{n}\rightarrow+\infty$ and let the eigenpairs
 $(\lambda_{p}(M_{\tau_{n}}),\varphi_{\tau_{n}})$ be defined by
\begin{equation}\label{411}
\begin{cases}
-\tau_{n}\partial_{t}\varphi_{\tau_{n}}+\mu\int_{\Omega}J(x-y)
\varphi_{\tau_{n}}(y,t)dy+b(x,t)\varphi_{\tau_{n}}+\lambda_{p}(M_{\tau_{n}})
\varphi_{\tau_{n}}=0\ \  \ \text{in }\ \bar{\Omega}\times[0,1],\\
\varphi_{\tau_{n}}\in\chi_{\Omega}^{++}, \quad ||\varphi_{\tau_{n}}||_{L^{2}(\Omega\times(0,1))}=1.
\end{cases}
\end{equation}

Multiplying equation \eqref{411} by $\varphi_{\tau_{n}}$ and integrating
over $\Omega\times(0,1)$, we get
\begin{equation*}
\begin{split}
\lambda_{p}(M_{\tau_{n}})&=\mu\int_{0}^{1}\int_{\Omega}\int_{\Omega}J(x-y)
\varphi_{\tau_{n}}(y,t)\varphi_{\tau_{n}}(x,t)dydxdt+
\int_{0}^{1}\int_{\Omega}b(x,t)\varphi_{\tau_{n}}^{2}(x,t)dxdt\\
&\leq\mu |\Omega| \underset{\mathbb{R^{N}}}{\max} J +\underset{\bar{\Omega}\times[0,1]}{\max} |b|.
\end{split}
\end{equation*}
Owing to the monotone non-decreasing of $\lambda_{p}(M_{\tau})$ on $\tau>0$, one gets
\begin{equation*}
\lim\limits_{n\rightarrow\infty}\lambda_{p}(M_{\tau_{n}})=\lambda_{p}
^{\infty}.
\end{equation*}

Multiplying equation \eqref{411} by $\partial_{t}\varphi_{\tau_{n}}$ and integrating over $\Omega\times(0,1)$ yield
\begin{equation*}
\begin{split}
\tau_{n}\int_{0}^{1}\int_{\Omega}|\partial_{t}\varphi_{\tau_{n}}|^{2}dxdt
&=\mu\int_{0}^{1}\int_{\Omega}\int_{\Omega}J(x-y)
\varphi_{\tau_{n}}(y,t)\partial_{t}\varphi_{\tau_{n}}(x,t)dydxdt\\
&\ \ \ +\int_{0}^{1}\int_{\Omega}\big[b(x,t)+\lambda_{p}(M_{\tau_{n}})\big]\varphi_{\tau_{n}}(x,t)
\partial_{t}\varphi_{\tau_{n}}(x,t)dxdt\\
&=\frac{1}{2}\int_{0}^{1}\int_{\Omega}\partial_{t}b(x,t)
\varphi_{\tau_{n}}^{2}(x,t)dxdt\\
&\leq\frac{1}{2}\underset{\bar{\Omega}\times[0,1]}{\max} |\partial_{t} b|,
\end{split}
\end{equation*}
which implies that
\begin{equation*}
||\partial_{t}\varphi_{\tau_{n}}||_{L^{2}(\Omega\times(0,1))}\rightarrow0
\ \ \ \text{as}\ \ n\rightarrow\infty.
\end{equation*}
Due to the above result and  $||\varphi_{\tau_{n}}||_{L^{2}}=1$, up to extraction,
there exists $w\in W^{1,2}((0,1);
L^{2}(\Omega))$ such that
\begin{equation*}
\varphi_{\tau_{n}}\rightharpoonup w \ \ \text{and}\ \ \partial_{t}\varphi_{\tau_{n}}\rightharpoonup \partial_{t}w.
\end{equation*}
Moreover, we have $||\partial_{t}w||_{L^{2}(\Omega\times(0,1))}\leq
\liminf\limits_{\tau\rightarrow0^{+}}||\partial_{t}\varphi_{\tau_{n}}||
_{L^{2}(\Omega\times(0,1))}=0$ and thus $w$ does not depend on $t$.

Passing to the limit $n\rightarrow\infty$ in \eqref{411}, we find that
$w$ is a weak solution of the  equation
\begin{equation*}
\mu\int_{\Omega}J(x-y)w(y)dy+b(x,t)w(x)+\lambda_{p}^{\infty}w(x)=0\ \ \  \text{in }\ \bar{\Omega}\times[0,1].
\end{equation*}
Integrating the above equation over $(0,1)$ yields
\begin{equation*}
\mu\int_{\Omega}J(x-y)w(y)dy+\hat{b}(x)w(x)+\lambda_{p}^{\infty}w(x)=0\ \ \  \text{in }\ \bar{\Omega}.
\end{equation*}
So $w\in C(\bar{\Omega})$ and $w>0$ in $\bar{\Omega}$, which
implies that $\lambda_{p}^{\infty}$ is the principal eigenvalue of
the operator $N_{\Omega}$. It is easy to know that $\lambda_{p}^{\infty}=\lambda_{p}(N_{\Omega})$. Thus, we have
\begin{equation*}
\lim\limits_{\tau\rightarrow\infty}\lambda_{p}(M_{\tau})=  \lambda_{p}(N_{\Omega}).
\end{equation*}

\textbf{Case 2.} If $\lambda_{p}(M_{\tau})$ is not a principal eigenvalue for some $\tau>0$, then we can use the approximation argument as in \textbf{Case 2} in the proof of Theorem \ref{th401}
to deduce the result.
\end{proof}

\subsection{Influences of the dispersal rate and the dispersal spread}
\noindent

In this subsection, we investigate the influences of the dispersal rate $\mu$ and the dispersal spread $\sigma$ on the generalised principal eigenvalue $\lambda_{p}(L_{\Omega}^{\tau,\mu,\sigma,m})$. Firstly, we establish 
the upper bound of the generalised principal eigenvalue $\lambda_{p}(L_{\Omega}^{\tau,\mu,\sigma,m})$.

\begin{lemma}\label{le204}
Assume that {\rm(J)}, {\rm(H)} and {\rm(A)} hold. Then
\begin{equation*}
\lambda_{p}(L_{\Omega}^{\tau,\mu,\sigma,m})\leq \underset{x\in \bar{\Omega}}{\min}\
\bigg\{\frac{\mu}{\sigma^{m}}h^{\sigma}(x)-\hat{a}(x)\bigg\}.
\end{equation*}
\end{lemma}

\begin{proof}
Fix $\lambda<\lambda_{p}(L_{\Omega}^{\tau,\mu,\sigma,m})$. By the definition
of $\lambda_{p}(L_{\Omega}^{\tau,\mu,\sigma,m})$, there exists $\varphi\in\chi_{\Omega}^{++}$ such that
\begin{equation*}
L_{\Omega}^{\tau,\mu,\sigma,m}[\varphi](x,t)+\lambda\varphi(x,t)\leq0\ \ \ \ \text{in}\
\bar{\Omega}\times\mathbb{R}.
\end{equation*}
It is easy to check that
\begin{equation*}
H_{\Omega}^{\tau,\mu,\sigma,m}[\varphi](x,t)+\lambda\varphi(x,t)\leq0\ \ \ \ \text{in}\
\bar{\Omega}\times\mathbb{R},
\end{equation*}
where $H_{\Omega}^{\tau,\mu,\sigma,m}[\varphi]=-\tau\varphi_{t}-\frac{\mu}{\sigma^{m}}h^{\sigma}
(x)\varphi+a(x,t)\varphi$. This implies that $\lambda\leq\lambda_{p}(H_{\Omega}^{\tau,\mu,\sigma,m})$. Thus
\begin{equation*}
\lambda_{p}(L_{\Omega}^{\tau,\mu,\sigma,m})\leq\lambda_{p}(H_{\Omega}^{\tau,\mu,\sigma,m}).
\end{equation*}
It follows from \cite[Propositions 3.4-3.5]{Rawal-2012-JDDE} that
\begin{equation*}
\lambda_{p}(H_{\Omega}^{\tau,\mu,\sigma,m})
=\underset{x\in\bar{\Omega}}{\min}\bigg\{
\frac{\mu}{\sigma^{m}}h^{\sigma}(x)-\hat{a}(x)\bigg\}.
\end{equation*}
This completes the proof.
\end{proof}

Now, we present the proof of
Theorem \ref{th102}.

\begin{proof}[Proof of Theorem {\rm\ref{th102}}]
(i) For any $\mu_{0}\in(0,\infty)$, applying Lemma \ref{le201}, we find that for each $\epsilon>0$, there exists $b_{\epsilon}\in C_{1}(\bar{\Omega}\times\mathbb{R})$ such that
$\underset{\bar{\Omega}\times\mathbb{R}}{\max}\ |b_{\epsilon}-b|<\epsilon$, 
$|\lambda_{1}-\lambda_{1}^{\epsilon}|<\epsilon$
and $\lambda_{1}^{\epsilon}$ is an isolated principal eigenvalue with finite multiplicity of $L_{\Omega}^{\tau,\mu_{0},1,0}(a_{\epsilon})$, where $b(x,t)=a(x,t)-\mu_{0} h^{1}(x)$, $a_{\epsilon}(x,t)=b_{\epsilon}(x,t)+\mu_{0} h^{1}(x)$. In
fact, we rewrite $L_{\Omega}^{\tau,\mu,1,0}(a_{\epsilon})$ as
\begin{equation*}
L_{\Omega}^{\tau,\mu,1,0}(a_{\epsilon})=L_{\Omega}^{\tau,\mu_{0},1,0}
(a_{\epsilon})+U_{\mu ,\mu_{0}},
\end{equation*}%
where%
\begin{equation*}
U_{\mu ,\mu_{0}}[\varphi ](x)=(\mu-\mu_{0})\bigg(\int_{\Omega
}J(x-y)\varphi (y)dy-h^{1}(x)\varphi(x)\bigg).
\end{equation*}%
Note that $U_{\mu ,\mu_{0}}$ is a linear bounded operator and $U_{\mu ,\mu_{0}}\rightarrow 0$ \ in norm as $\mu \rightarrow \mu_{0}$.
It follows from the classical perturbation theory of isolated eigenvalues \cite[Charpter 4, Section 3.5]{Kato-1995}, there exists $\delta _{0}>0$ such that for all $|\mu-\mu_{0}|<\delta
_{0}$, we have%
\begin{equation*}
\big|\lambda _{1}(L_{\Omega}^{\tau,\mu,1,0}(a_{\epsilon}))-\lambda _{1}(L_{\Omega}^{\tau,\mu_{0},1,0}(a_{\epsilon}))\big|<\epsilon.
\end{equation*}
Thanks to Theorem \ref{th101}, we obtain
\begin{equation*}
\big|\lambda _{p}(L_{\Omega}^{\tau,\mu,1,0}(a_{\epsilon}))-\lambda _{p}(L_{\Omega}^{\tau,\mu_{0},1,0}(a_{\epsilon}))\big|<\epsilon.
\end{equation*}

By Lemma \ref{le202}, $\lambda _{p}(L_{\Omega}^{\tau,\mu,1,0}(a))$
is Lipschitz continuous with respect to $a$, i.e.,%
\begin{equation*}
\big|\lambda _{p}(L_{\Omega}^{\tau,\mu,1,0}(a_{\epsilon }))-\lambda _{p}(L_{\Omega}^{\tau,\mu,1,0}(a))\big|\leq ||a-a_{\epsilon }||_{L^{\infty }(\Omega )}<
\epsilon.
\end{equation*}

In a word, for every given constant $\epsilon >0$, there exist $\delta
_{0}>0 $ and $a_{\epsilon }\in C_{1}(\overline{\Omega }\times \mathbb{R})$ such that for all $%
|\mu -\mu_{0}|<\delta _{0}$, we have%
\begin{align*}
&|\lambda _{p}(L_{\Omega}^{\tau,\mu,1,0}(a))-\lambda _{p}(L_{\Omega}^{\tau,\mu_{0},1,0}(a))| \\
\leq &|\lambda _{p}(L_{\Omega}^{\tau,\mu,1,0}(a))-\lambda _{p}(L_{\Omega}^{\tau,\mu,1,0}(a_{\epsilon }))|+|\lambda _{p}(L_{\Omega}^{\tau,\mu,1,0}(a_{\epsilon }))-\lambda _{p}(L_{\Omega}^{\tau,\mu_{0},1,0}(a_{\epsilon }))| \\
&+|\lambda _{p}(L_{\Omega}^{\tau,\mu_{0},1,0}(a_{\epsilon }))-\lambda
_{p}(L_{\Omega}^{\tau,\mu_{0},1,0}(a))| \\
<&\epsilon+\epsilon+\epsilon=3\epsilon.
\end{align*}%
So $\lambda _{p}(L_{\Omega}^{\tau,\mu,1,0})$ is continuous with respect to $\mu$.

Now, we prove the asymptotic behavior of $\lambda_{p}(L_{\Omega}^{\tau,\mu,1,0})$ as $\mu\rightarrow0^{+}$.
For simplicity, we write $\lambda_{p}^{\mu}:=\lambda_{p}(L_{\Omega}^{\tau,\mu,1,0})$. We first claim that
for each $\epsilon>0$, there is $\mu_{\epsilon}>0$ such that
\begin{equation}\label{209}
\lambda_{\epsilon}^{max}\leq \lambda_{p}^{\mu}\leq \lambda_{\epsilon}^{min}\ \ \ \ \text{for all} \ \mu\in (0,\mu_{\epsilon}),
\end{equation}
where $\lambda_{\epsilon}^{max}=-\max_{\bar{\Omega}} \hat{a}-\epsilon$,
$\lambda_{\epsilon}^{min}=-\min_{\bar{\Omega}} \hat{a}+\epsilon$.
In fact, it is easy to check that the function $\phi(x,t):=e^{\int_{0}^{t}[a(x,s)-\hat{a}(x)]ds}$ is a positive 1-periodic solution of $\phi_{t}=a(x,t)\phi-\hat{a}(x)\phi$ for $(x,t)\in \bar{\Omega}\times\mathbb{R}$. In particular, $\phi\in\chi_{\Omega}^{++}$. A simple computation yields
\begin{align*}
(L_{\Omega}^{\tau,\mu,1,0}+\lambda_{\epsilon}^{max})[\phi](x,t)=&\mu\bigg(\int_{\Omega}
J(x-y)\phi(y,t)dy-h^{1}(x)\phi(x,t)\bigg)\\
&+\bigg(\hat{a}(x)-\underset{\bar{\Omega}}{\max}\ \hat{a}-\epsilon\bigg)\phi(x,t),\\
(L_{\Omega}^{\tau,\mu,1,0}+\lambda_{\epsilon}^{min})[\phi](x,t)=&\mu\bigg(\int_{\Omega}
J(x-y)\phi(y,t)dy-h^{1}(x)\phi(x,t)\bigg)\\
&+\bigg(\hat{a}(x)-\underset{\bar{\Omega}}{\min}\ \hat{a}+\epsilon \bigg)\phi(x,t).
\end{align*}
Thus, there exists $\mu_{\epsilon}>0$ such that
\begin{equation*}
(L_{\Omega}^{\tau,\mu,1,0}+\lambda_{\epsilon}^{max})[\phi]\leq0\ \  \text{and}\ \
(L_{\Omega}^{\tau,\mu,1,0}+\lambda_{\epsilon}^{min})[\phi]\geq0\ \ \ \ \text{for all}\ \ \mu\in (0,\mu_{\epsilon}).
\end{equation*}
Moreover, by Theorem \ref{th101} and the definitions of $\lambda_{p}(L_{\Omega}^{\tau,\mu,1,0})$ and $\lambda_{p}^{'}(L_{\Omega}^{\tau,\mu,1,0})$,
there holds (\ref{209}) for all $\mu\in (0,\mu_{\epsilon})$.

Next, thanks to Lemma \ref{le204} and the inequality \eqref{209}, for each $\epsilon>0$ there is $\mu_{\epsilon}>0$ such that
\begin{equation*}
-\max_{\bar{\Omega}} \hat{a}-\epsilon\leq \lambda_{p}^{\mu}\leq \underset{x\in \bar{\Omega}}{\min} \bigg[\mu h^{1}(x)-\hat{a}(x)\bigg]\ \ \ \ \text{for all} \ \mu\in (0,\mu_{\epsilon}).
\end{equation*}
Passing $\mu\rightarrow0^{+}$, we find
\begin{equation*}
-\max_{\bar{\Omega}} \hat{a}-\epsilon\leq \underset{\mu\rightarrow0^{+}}{\liminf} \lambda_{p}^{\mu}\leq
\underset{\mu\rightarrow0^{+}}{\limsup} \lambda_{p}^{\mu}\leq
-\max_{\bar{\Omega}} \hat{a}\ \ \ \ \ \forall \ \epsilon>0,
\end{equation*}
which leads to
\begin{equation*}
\lambda_{p}^{\mu}\rightarrow -\max_{\bar{\Omega}} \hat{a}\ \ \ \ \ \text{as} \ \ \mu\rightarrow0^{+}.
\end{equation*}
(ii) The proof of continuity is similar to (i). Hence, we omit it.
It remains to prove parts (a) and (b).\\
(a) By Lemma \ref{le204}, we have
\begin{equation*}
\lambda_{p}(L_{\Omega}^{\tau,\mu,\sigma,m})\leq \underset{x\in \bar{\Omega}}{\min}\
\bigg\{\frac{\mu}{\sigma^{m}}h^{\sigma}(x)-\hat{a}(x)\bigg\}.
\end{equation*}
As $m>0$ and $\sigma\rightarrow\infty$, there holds
\begin{equation*}
\frac{\mu}{\sigma^{m}}h^{\sigma}(x)
\rightarrow0
\ \ \ \text{for all} \ x\in\bar{\Omega},
\end{equation*}
which implies that
\begin{equation*}
\underset{\sigma\rightarrow\infty}{\limsup}\ \lambda_{p}(L_{\Omega}^{\tau,\mu,\sigma,m})\leq-\underset{\bar{\Omega}}{\max}\ \hat{a}.
\end{equation*}

To complete our proof, it remains to show
\begin{equation*}
-\underset{\bar{\Omega}}{\max}\ \hat{a}\leq \underset{\sigma\rightarrow\infty}{\liminf}\ \lambda_{p}(L_{\Omega}^{\tau,\mu,\sigma,m}).
\end{equation*}
For fixed constant $\phi_{0}>0$, it is easy to check that for every $x\in\bar{\Omega}$, the function
\begin{equation}\label{223}
\phi(x,t)=e^{\int_{0}^{t}(a(x,s)-\hat{a}(x))ds}\phi_{0},\ \ \ \ t\in\mathbb{R},
\end{equation}
is a positive 1-periodic solution of the ordinary differential equation
$v_{t}=a(x,t)v-\hat{a}(x)v$ with the initial condition $v(x,0)=\phi_{0}$.
In particular, $\phi\in\chi_{\Omega}^{++}$ and we can choose $\phi_{0}$
such that $\sup_{\bar{\Omega}\times\mathbb{R}} \phi=1$. For
every $\epsilon>0$, we have
\begin{align*}
&(L_{\Omega}^{\tau,\mu,\sigma,m}-\underset{\bar{\Omega}}{\max}\ \hat{a}-\epsilon)[\phi](x,t)\\
=&-\tau\phi_{t}(x,t)+\bigg(a(x,t)-\underset{\bar{\Omega}}{\max}\ \hat{a}-\epsilon\bigg)\phi(x,t)\\
&\ \ \ \ \ \ \ +\frac{\mu}{\sigma^{m}}\bigg(\int_{\Omega}J_{\sigma}(x-y)
\phi(y,t)dy-h^{\sigma}(x)\phi(x,t)\bigg)\\
\leq&\frac{\mu}{\sigma^{m}}\bigg(\int_{\Omega}J_{\sigma}(x-y)
\phi(y,t)dy-h^{\sigma}(x)\phi(x,t)\bigg)-\epsilon\phi(x,t).
\end{align*}
Using $\sup_{\bar{\Omega}\times\mathbb{R}} \phi=1$, there holds
\begin{align*}
&\bigg\|\frac{\mu}{\sigma^{m}}\bigg(\int_{\Omega}J_{\sigma}(x-y)
\phi(y,t)dy-h^{\sigma}(x)\phi(x,t)\bigg)\bigg\|_{\infty}\\
\leq&\bigg\|\frac{\mu}{\sigma^{m}}\int_{\Omega}J_{\sigma}(x-y)
dy\bigg\|_{\infty}+\frac{\mu M}{\sigma^{m}}\rightarrow0 \ \ \ \ \text{as}\ \sigma\rightarrow\infty,
\end{align*}
which implies that there is $\sigma_{\epsilon}>0$ such that
\begin{equation*}
(L_{\Omega}^{\tau,\mu,\sigma,m}-\underset{\bar{\Omega}}{\max}\ \hat{a}-\epsilon)[\phi]\leq0 \ \ \ \ \text{for all}\ \sigma\geq\sigma_{\epsilon}.
\end{equation*}
Thanks to the definition of $\lambda_{p}(L_{\Omega}^{\tau,\mu,\sigma,m})$, there holds
\begin{equation*}
\lambda_{p}(L_{\Omega}^{\tau,\mu,\sigma,m})\geq
-\underset{\bar{\Omega}}{\max}\ \hat{a}-\epsilon\ \ \ \ \text{for all}\ \sigma\geq\sigma_{\epsilon}.
\end{equation*}
Since $\epsilon$ is an arbitrary constant, we have
\begin{equation*}
 \underset{\sigma\rightarrow\infty}{\liminf}\ \lambda_{p}(L_{\Omega}^{\tau,\mu,\sigma,m})\geq-\underset{\bar{\Omega}}
 {\max}\ \hat{a}.
\end{equation*}
Thus, we get
\begin{equation*} \underset{\sigma\rightarrow\infty}{\lim}\ \lambda_{p}(L_{\Omega}^{\tau,\mu,\sigma,m})=-\underset{\bar{\Omega}}{\max}\ \hat{a}.
\end{equation*}
(b) Again, following from Lemma \ref{le204}, we have
\begin{equation*}
\lambda_{p}(L_{\Omega}^{\tau,\mu,\sigma,0})\leq \underset{x\in \bar{\Omega}}{\min}\
\bigg\{\mu h^{\sigma}(x)-\hat{a}(x)\bigg\}.
\end{equation*}
Owing to $\underset{\sigma\rightarrow\infty}{\lim}\ h^{\sigma}(x)=c$,
this implies that
\begin{equation*}
\underset{\sigma\rightarrow\infty}{\limsup}\ \lambda_{p}(L_{\Omega}^{\tau,\mu,\sigma,0})\leq\mu c-\underset{\bar{\Omega}}{\max}\ \hat{a}.
\end{equation*}

To complete our proof, it remains to obtain
\begin{equation*}
\mu c-\underset{\bar{\Omega}}{\max}\ \hat{a}\leq \underset{\sigma\rightarrow\infty}{\liminf}\ \lambda_{p}(L_{\Omega}^{\tau,\mu,\sigma,0}).
\end{equation*}
For fixed constant $\phi_{0}>0$, it is easy to check that for every $x\in\bar{\Omega}$, the function
\begin{equation}\label{523}
\phi(x,t)=e^{\int_{0}^{t}(a(x,s)-\hat{a}(x))ds}\phi_{0},\ \ \ \ t\in\mathbb{R},
\end{equation}
is a positive 1-periodic solution of the ordinary differential equation
$v_{t}=a(x,t)v-\hat{a}(x)v$ with the initial condition $v(x,0)=\phi_{0}$.
In particular, $\phi\in\chi_{\Omega}^{++}$ and we can choose $\phi_{0}$
such that $\sup_{\bar{\Omega}\times\mathbb{R}} \phi=1$. For
every $\epsilon>0$, we have
\begin{align*}
&(L_{\Omega}^{\tau,\mu,\sigma,0}-\underset{\bar{\Omega}}{\max}\ \hat{a}+\mu c-\epsilon)[\phi](x,t)\\
=&-\tau\phi_{t}(x,t)+\bigg(a(x,t)-\underset{\bar{\Omega}}{\max}\ \hat{a}-\epsilon\bigg)\phi(x,t)\\
&\ \ \ \ \ \ \ +\mu \bigg(\int_{\Omega}J_{\sigma}(x-y)
\phi(y,t)dy-h^{\sigma}(x)\phi(x,t)+c\phi(x,t)\bigg)\\
\leq& \mu \bigg(\int_{\Omega}J_{\sigma}(x-y)
\phi(y,t)dy-h^{\sigma}(x)\phi(x,t)+c\phi(x,t)\bigg)-\epsilon\phi(x,t).
\end{align*}
Using $\sup_{\bar{\Omega}\times\mathbb{R}} \phi=1$, there holds
\begin{align*}
&\bigg\|\mu \bigg(\int_{\Omega}J_{\sigma}(x-y)
\phi(y,t)dy-h^{\sigma}(x)\phi(x,t)+c \phi(x,t)\bigg)\bigg\|_{\infty}\\
\leq&\bigg\|\mu \int_{\Omega}J_{\sigma}(x-y)
dy\bigg\|_{\infty}+\|c-h^{\sigma}\|_{\infty}\rightarrow0 \ \ \ \ \text{as}\ \sigma\rightarrow\infty,
\end{align*}
which implies that there is $\sigma_{\epsilon}>0$ such that
\begin{equation*}
(L_{\Omega}^{\tau,\mu,\sigma,0}-\underset{\bar{\Omega}}{\max}\ \hat{a}+\mu c-\epsilon)[\phi]\leq0 \ \ \ \ \text{for all}\ \sigma\geq\sigma_{\epsilon}.
\end{equation*}
Thanks to the definition of $\lambda_{p}(L_{\Omega}^{\tau,\mu,\sigma,0})$, there holds
\begin{equation*}
\lambda_{p}(L_{\Omega}^{\tau,\mu,\sigma,0})\geq
\mu c-\underset{\bar{\Omega}}{\max}\ \hat{a}-\epsilon\ \ \ \ \text{for all}\ \sigma\geq\sigma_{\epsilon}.
\end{equation*}
Since $\epsilon$ is an arbitrary constant, we have
\begin{equation*}
 \underset{\sigma\rightarrow\infty}{\liminf}\ \lambda_{p}(L_{\Omega}^{\tau,\mu,\sigma,0})\geq\mu c-\underset{\bar{\Omega}}{\max} \hat{a}.
\end{equation*}
Thus, we get
\begin{equation*} \underset{\sigma\rightarrow\infty}{\lim}\ \lambda_{p}(L_{\Omega}^{\tau,\mu,\sigma,0})=\mu c-\underset{\bar{\Omega}}{\max}\ \hat{a}.
\end{equation*}
This completes the proof of Theorem \ref{th102}.
\end{proof}

Next, we recall the following lemma in \cite[Corollary D]{Rawal-2012-JDDE}.

\begin{lemma}\label{le203}
Assume that {\rm(J)} and {\rm(A)} hold. Let $N_{\Omega}^{\mu,\sigma,m}[\varphi](x):=\frac{\mu}{\sigma^{m}}\int_{\Omega}
J_{\sigma}(x-y)(\varphi(y)-\varphi(x))dy+\hat{a}(x)\varphi(x)$. If $\lambda_{p}(N_{\Omega}^{\mu,\sigma,m})$ is the principal eigenvalue of $N_{\Omega}^{\mu,\sigma,m}$ and $h^{\sigma}(x)=
\int_{\Omega}J_{\sigma}(x-y)dy$, then $\lambda_{p}(L_{\Omega}^{\tau,\mu,\sigma,m})$ is the principal eigenvalue of $L_{\Omega}^{\tau,\mu,\sigma,m}$.
\end{lemma}

Finally, we prove Theorem \ref{th103}.

\begin{proof}[Proof of Theorem {\rm\ref{th103}}]
(i) By \cite[Theorem 2.2 (3)]{Shen-2015-DCDS}, there exists $\mu_{1}>0$
such that $\lambda_{p}(N_{\Omega}^{\mu,1,0})$ is the principal eigenvalue of $N_{\Omega}^{\mu,1,0}$ for all $\mu\geq\mu_{1}$. It follows from Lemma \ref{le203} that $\lambda_{p}(L_{\Omega}^{\tau,\mu,1,0})$ is the principal eigenvalue of $L_{\Omega}^{\tau,\mu,1,0}$ for all $\mu\geq\mu_{1}$.

Since $\lambda_{p}^{\mu}$ is the principal eigenvalue of $L_{\Omega}^{\tau,\mu,1,0}$ for all $\mu\geq \mu_{1}$, there exists $\varphi\in\chi_{\Omega}^{++}$ with $\int_{0}^{1}\int_{\Omega}\varphi^{2}(x,t)dxdt=1$
such that
\begin{equation} \label{210}
-\tau\varphi_{t}(x,t)+\mu\int_{\Omega}J(x-y)(\varphi(y,t)-\varphi(x,t))dy
+(a(x,t)+\lambda_{p}^{\mu})\varphi(x,t)=0\ \ \ \text{in} \ \bar{\Omega}\times\mathbb{R}.
\end{equation}

On one hand, divide (\ref{210}) by $\varphi$ and integrate over $\Omega\times(0,1)$
to obtain
\begin{align*}
\int_{0}^{1}\int_{\Omega} \lambda_{p}^{\mu}dxdt= &\tau\int_{0}^{1}\int_{\Omega} \frac{\varphi_{t}(x,t)}{\varphi(x,t)}dxdt
-\int_{0}^{1}\int_{\Omega}a(x,t)dxdt\\
&-\mu \int_{0}^{1}\int_{\Omega}\int_{\Omega}J(x-y) \frac{\varphi(y,t)-\varphi(x,t)}{\varphi(x,t)}dydxdt.
\end{align*}
Owing to
\begin{equation*}
\int_{0}^{1}\int_{\Omega} \frac{\varphi_{t}(x,t)}{\varphi(x,t)}dxdt
=\int_{\Omega}\int_{0}^{1} \frac{\varphi_{t}(x,t)}{\varphi(x,t)}dtdx=0
\end{equation*}
and
\begin{align*}
&\mu \int_{0}^{1}\int_{\Omega}\int_{\Omega}J(x-y) \frac{\varphi(y,t)-\varphi(x,t)}{\varphi(x,t)}dydxdt\\
=&\frac{\mu}{2} \int_{0}^{1}\int_{\Omega}\int_{\Omega}J(x-y) \frac{(\varphi(y,t)-\varphi(x,t))^{2}}{\varphi(x,t)\varphi(y,t)}dydxdt,
\end{align*}
we get
\begin{align*}
&\int_{0}^{1}\int_{\Omega} \lambda_{p}^{\mu}dxdt\\
=&-\int_{0}^{1}\int_{\Omega}a(x,t)dxdt
-\frac{\mu}{2} \int_{0}^{1}\int_{\Omega}\int_{\Omega}J(x-y) \frac{(\varphi(y,t)-\varphi(x,t))^{2}}{\varphi(x,t)\varphi(y,t)}dydxdt.
\end{align*}
This implies that
\begin{equation}\label{211}
\lambda_{p}^{\mu}\leq -\bar{\hat{a}}.
\end{equation}

On the other hand, multiplying (\ref{210}) by $\varphi$ and integrating
over $\Omega\times(0,1)$ yield
\begin{align*}
\int_{0}^{1}\int_{\Omega} \lambda_{p}^{\mu}\varphi^{2}(x,t)dxdt= &\tau\int_{0}^{1}\int_{\Omega} \varphi_{t}(x,t)\varphi(x,t)dxdt
-\int_{0}^{1}\int_{\Omega}a(x,t)\varphi^{2}(x,t)dxdt\\
&-\mu \int_{0}^{1}\int_{\Omega}\int_{\Omega}J(x-y) (\varphi(y,t)-\varphi(x,t))\varphi(x,t)dydxdt.
\end{align*}
In view of $\varphi\in \chi_{\Omega}^{++}$ and the symmetry of $J$,
we have
\begin{align}\label{212}
\lambda_{p}^{\mu}=
\frac{\mu}{2} \int_{0}^{1}\int_{\Omega}\int_{\Omega}J(x-y) (\varphi(y,t)-\varphi(x,t))^{2}dydxdt
-\int_{0}^{1}\int_{\Omega}a(x,t)\varphi^{2}(x,t)dxdt,
\end{align}
which implies that
\begin{equation}\label{213}
\lambda_{p}^{\mu}\geq -\int_{0}^{1}\int_{\Omega}a(x,t)\varphi^{2}(x,t)dxdt\geq
-\underset{\bar{\Omega}\times[0,1]}{\max}\ a.
\end{equation}
By combining (\ref{211}) and (\ref{213}), we obtain
\begin{equation}\label{214}
-\underset{\bar{\Omega}\times[0,1]}{\max}\ a\leq\lambda_{p}^{\mu}
\leq -\bar{\hat{a}}.
\end{equation}

Rewriting (\ref{212}) as
\begin{align*}
&\int_{0}^{1}\int_{\Omega}\int_{\Omega}J(x-y) (\varphi(y,t)-\varphi(x,t))^{2} dydxdt\\
=&\frac{2}{\mu} \int_{0}^{1}\int_{\Omega}(\lambda_{p}^{\mu}+a(x,t))\varphi^{2}(x,t)dxdt,
\end{align*}
it follows from (\ref{214}) that
\begin{equation}\label{215}
\begin{split}
&\int_{0}^{1}\int_{\Omega}\int_{\Omega}J(x-y) (\varphi(y,t)-\varphi(x,t))^{2} dydxdt\\
\leq& \frac{2}{\mu}\bigg(\underset{\bar{\Omega}\times[0,1]}{\max}\ a
-\bar{\hat{a}}\bigg).
\end{split}
\end{equation}
Let $\psi(x,t):=\varphi(x,t)-\bar{\varphi}(t)$, where
$\bar{\varphi}(t)=\frac{1}{|\Omega|}\int_{\Omega}\varphi(x,t)dx$.
Then we have $\int_{\Omega}\psi(x,t)dx=0$.
Observe that
\begin{equation}\label{217}
\int_{\Omega}\int_{\Omega}J(x-y) (\varphi(y,t)-\varphi(x,t))^{2} dydx
= \int_{\Omega}\int_{\Omega}J(x-y) (\psi(y,t)-\psi(x,t))^{2} dydx.
\end{equation}
By \cite[Page 1688, Formula (5.6)]{Shen-2015-DCDS}, there exists $C>0$ such that
\begin{equation}\label{218}
\int_{\Omega}\psi^{2}(x,t)dx\leq\frac{1}{2C}\int_{\Omega}\int_{\Omega}J(x-y) (\psi(y,t)-\psi(x,t))^{2} dydx\ \ \ \text{for all}~~ \mu\gg1.
\end{equation}
It deduces from (\ref{215}),
(\ref{217}) and (\ref{218}) that
\begin{equation}\label{219}
\begin{split}
\int_{0}^{1}\int_{\Omega}\psi^{2}(x,t)dxdt&\leq
\frac{1}{2C}\int_{0}^{1}\int_{\Omega}\int_{\Omega}J(x-y) (\psi(y,t)-\psi(x,t))^{2} dydxdt\\
&\leq\frac{1}{C\mu}\bigg(\underset{\bar{\Omega}\times[0,1]}{\max}\ a
-\bar{\hat{a}}\bigg).
\end{split}
\end{equation}

Now, integrating (\ref{210}) over $\Omega$ and substituting
$\psi=\varphi-\bar{\varphi}$ yield
\begin{equation}\label{221}
\tau\bar{\varphi}_{t}=\frac{1}{|\Omega|}\int_{\Omega}(\lambda_{p}^{\mu}+a(x,t))dx\bar{\varphi}
+\frac{1}{|\Omega|}\int_{\Omega}(\lambda_{p}^{\mu}+a(x,t))\psi(x,t)dx.
\end{equation}
In view of (\ref{219}), we have
\begin{equation*}
\int_{0}^{1}\int_{\Omega}(\lambda_{p}^{\mu}+a(x,t))\psi(x,t)dxdt
=O(\mu^{-\frac{1}{2}})\ \ \ \ \text{as} \ \mu\rightarrow\infty.
\end{equation*}
Using the integrating form in this first order differential equation,
we find that
\begin{equation}\label{222}
\bar{\varphi}(t)=\bar{\varphi}(0)e^{\frac{1}{\tau|\Omega|}\int_{0}^{t}\int_{\Omega}
(\lambda_{p}^{\mu}+a(x,t))dxdt}+O(\mu^{-\frac{1}{2}})\ \ \ \ \text{as} \ \mu\rightarrow\infty.
\end{equation}
Since $\bar{\varphi}(1)=\bar{\varphi}(0)$, we get
\begin{equation*}
\int_{0}^{1}\int_{\Omega}
(\lambda_{p}^{\mu}+a(x,t))dxdt\rightarrow0 \ \ \text{or}\ \ \bar{\varphi}(0)\rightarrow0\ \ \ \ \text{as} \ \mu\rightarrow\infty.
\end{equation*}
If $\bar{\varphi}(0)\rightarrow0$, then, by (\ref{222}), also $\bar{\varphi}(t)\rightarrow0$ uniformly in $t\in [0,1]$ as
$\mu\rightarrow\infty$.
Thanks to (\ref{215}) and the symmetry of $J$, we get
\begin{align*}
&\int_{0}^{1}\int_{\Omega}\varphi^{2}(x,t)dxdt\\
\leq& C_{0}\int_{0}^{1}\int_{\Omega}\int_{\Omega}J(x-y)\varphi^{2}(x,t)dydxdt\\
=&C_{0}\int_{0}^{1}\int_{\Omega}\int_{\Omega}J(x-y)(\varphi^{2}(x,t)-
\varphi(y,t)\varphi(x,t))dydxdt\\
&+C_{0}\int_{0}^{1}\int_{\Omega}\int_{\Omega}J(x-y)
\varphi(y,t)\varphi(x,t)dydxdt\\
\leq&\frac{C_{0}}{2}\int_{0}^{1}\int_{\Omega}\int_{\Omega}J(x-y)
(\varphi(y,t)-\varphi(x,t))^{2}dydxdt\\
&+|\Omega|^{2}C_{0}M \int_{0}^{1}\bar{\varphi}^{2}(t)dt\\
\leq&\frac{C_{0}}{\mu}\bigg(\underset{\bar{\Omega}\times[0,1]}{\max}\ a
-\bar{\hat{a}}\bigg)+|\Omega|^{2}C_{0}M \int_{0}^{1}\bar{\varphi}^{2}(t)dt,
\end{align*}
where $C_{0}=\bigg(\underset{x\in \bar{\Omega}}{\min}\int_{\Omega}J(x-y)dy\bigg)^{-1}$ and $M=\underset{(x,y)\in \bar{\Omega}\times\bar{\Omega}}{\max}J(x-y)$.
This implies that
\begin{equation*}
\int_{0}^{1}\int_{\Omega}\varphi^{2}(x,t)dxdt\rightarrow0\ \ \ \ \text{as} \ \mu\rightarrow\infty,
\end{equation*}
in contradiction to the normalization of $\varphi$. Thus, we have
\begin{equation*}
\int_{0}^{1}\int_{\Omega}
(\lambda_{p}^{\mu}+a(x,t))dxdt\rightarrow0 \ \ \ \ \text{as} \ \mu\rightarrow\infty,
\end{equation*}
that is,
\begin{equation*}
\lambda_{p}^{\mu}\rightarrow-\bar{\hat{a}} \ \ \ \ \text{as} \ \mu\rightarrow\infty.
\end{equation*}
(ii) By \cite[Theorem 1.2]{Su-2019-submitted}, there exists $\sigma_{0}>0$
such that $\lambda_{p}(N_{\Omega}^{\mu,\sigma,m})$ is the principal eigenvalue of $N_{\Omega}^{\mu,\sigma,m}$ for all $\sigma\leq\sigma_{0}$. It follows from Lemma \ref{le203} that $\lambda_{p}(L_{\Omega}^{\tau,\mu,\sigma,m})$ is the principal eigenvalue of $L_{\Omega}^{\tau,\mu,\sigma,m}$ for all $\sigma\leq\sigma_{0}$.

Let $\phi$ be as in \eqref{223}. Without loss of generality, we assume that $a\in C_{1}^{2,1}(\bar{\Omega}\times\mathbb{R})$. Then there holds $\phi\in C_{1}^{2,1}(\bar{\Omega}\times\mathbb{R})\cap\chi_{\Omega}^{++}$.
For every $\epsilon>0$, we have
\begin{equation}\label{224}
\begin{split}
&(L_{\Omega}^{\tau,\mu,\sigma,m}-\underset{\bar{\Omega}}{\max}\ \hat{a}-\epsilon)[\phi](x,t)\\
\leq&\frac{\mu}{\sigma^{m}}\int_{\Omega}J_{\sigma}(x-y)
(\phi(y,t)-\phi(x,t))dy-\epsilon\phi(x,t)\\
\leq&\frac{\mu}{\sigma^{m}}\int_{\frac{\Omega-x}{\sigma}}J(z)
(\phi(x+\sigma z,t)-\phi(x,t))dz-\epsilon\phi(x,t).
\end{split}
\end{equation}
For $\sigma$ small enough, say $\sigma\leq\sigma_{1}$, we obtain
$B(0,1)\subset\frac{\Omega-x}{\sigma}$ for all $x\in \Omega$. Thus,
by Taylor's expansion and the symmetric of $J$, there holds
\begin{equation}\label{225}
\begin{split}
&\frac{\mu}{\sigma^{m}}\int_{\frac{\Omega-x}{\sigma}}J(z)
(\phi(x+\sigma z,t)-\phi(x,t))dz\\
=&\frac{\mu}{\sigma^{m}}\int_{\mathbb{R}^{N}}J(z)
(\phi(x+\sigma z,t)-\phi(x,t))dz\\
=&\frac{\mu}{\sigma^{m}}\int_{\mathbb{R}^{N}}J(z)
\bigg(D\phi(x,t)(\sigma z)+\frac{1}{2}(\sigma z)^{T}D^{2}\phi(x,t)(\sigma z)+o(\sigma^{2})\bigg)dz\\
=&\frac{\mu\sigma^{2-m}}{2}\int_{\mathbb{R}^{N}}J(z)
z^{T}D^{2}\phi(x,t)zdz+o(\sigma^{2-m}).
\end{split}
\end{equation}
By combining (\ref{224}) with (\ref{225}), there exists $0<\sigma_{\epsilon}<\sigma_{1}$ such that
\begin{equation*}
(L_{\Omega}^{\tau,\mu,\sigma,m}-\underset{\bar{\Omega}}{\max}\ \hat{a}-\epsilon)[\phi]\leq0 \ \ \ \ \text{for all}\ \sigma\leq\sigma_{\epsilon}.
\end{equation*}
Using the definition of $\lambda_{p}(L_{\Omega}^{\tau,\mu,\sigma,m})$, there holds
\begin{equation*}
\lambda_{p}(L_{\Omega}^{\tau,\mu,\sigma,m})\geq
-\underset{\bar{\Omega}}{\max}\ \hat{a}-\epsilon\ \ \ \ \text{for all}\ \sigma\leq\sigma_{\epsilon}.
\end{equation*}
Therefore, we obtain
\begin{equation*}
\underset{\sigma\rightarrow0^{+}}{\liminf}\ \lambda_{p}(L_{\Omega}^{\tau,\mu,\sigma,m})\geq-\underset{\bar{\Omega}}{\max}\ \hat{a}.
\end{equation*}

It remains to show that
\begin{equation*}
\underset{\sigma\rightarrow0^{+}}{\limsup}\ \lambda_{p}(L_{\Omega}^{\tau,\mu,\sigma,m})\leq-\underset{\bar{\Omega}}{\max}\ \hat{a}.
\end{equation*}
By Theorem \ref{th101} and \cite[Theorem C]{Rawal-2012-JDDE}, it follows that
\begin{equation}\label{226}
\lambda_{p}(L_{\Omega}^{\tau,\mu,\sigma,m})\leq\lambda_{p}(N_{\Omega}^{\mu,\sigma,m}).
\end{equation}
Using \cite[Theorem 1.2]{Su-2019-submitted}, there holds
\begin{equation}\label{227}
\underset{\sigma\rightarrow0^{+}}{\lim}\ \lambda_{p}(N_{\Omega}^{\mu,\sigma,m})=-\underset{\bar{\Omega}}{\max}\ \hat{a}.
\end{equation}
In view of (\ref{226}) and (\ref{227}), we have
\begin{equation*}
\underset{\sigma\rightarrow0^{+}}{\limsup}\ \lambda_{p}(L_{\Omega}^{\tau,\mu,\sigma,m})\leq-\underset{\bar{\Omega}}{\max}\ \hat{a}.
\end{equation*}
The proof is complete.
\end{proof}


\section{Time-periodic nonlocal dispersal KPP equations}

In this section we apply
the results for the 
generalised principal eigenvalues 
to the time-periodic nonlocal dispersal KPP equation with Neumann boundary conditions. Firstly, we
study the effects of the frequency on the persistence of populations.
Next, we discuss the effects of the dispersal rate and the dispersal spread on the positive time-periodic solutions to the equation. More precisely, we obtain the existence, uniqueness and stability of positive time-periodic solutions when $\mu$ or $\sigma $ is sufficiently small or large. Furthermore, we analyse the asymptotic limits of the positive time-periodic solutions as $\mu$ or $\sigma$ tends to zero or infinity.

\subsection{Effects of the frequency}
\noindent

The following result provides 
some sufficient
conditions such that Corollary \ref{17} holds.

\begin{theorem}\label{pro301}
Assume that {\rm(J)} and {\rm(F)} hold. Then the following statements hold:
\begin{itemize}
\item[(i)] If $\int_{0}^{1}\underset{x\in\bar{\Omega}}{\max} \{a(x,t)\}dt<0$, then  \eqref{105} admits no solution in $\chi_{\Omega}^{+}\setminus\{0\}$ and zero solution is globally asymptotically stable for all $\tau\in(0,\infty)$;
\item[(ii)] If $\int_{0}^{1}\underset{x\in\bar{\Omega}}{\max} \{a(x,t)\}dt-\frac{\mu}{\sigma^{m}}>0$, $\underset{x\in\bar{\Omega}}{\max} \{\hat{a}(x)\}<0$ and $\lambda_{p}(L_{\Omega}^{\tau,\mu,\sigma,m})$ is a principal eigenvalue of the operator $L_{\Omega}^{\tau,\mu,\sigma,m}$, then there is some constant $\tau^{*}>0$ such that
    \begin{itemize}
    \item[(a)] If $\tau<\tau^{*}$, then \eqref{105} admits a unique solution $u_{\tau}^{\ast}\in \chi_{\Omega}^{++}$ that is globally asymptotically stable.
    \item[(b)] If $\tau>\tau^{*}$, then \eqref{105} admits no solution in $\chi_{\Omega}^{+}\setminus\{0\}$ and zero solution is globally asymptotically stable;
    \end{itemize}
\item[(iii)] If $\underset{x\in\bar{\Omega}}{\max} \{\hat{a}(x)\}-\frac{\mu}{\sigma^{m}}>0$, then  \eqref{105} admits a unique solution $u_{\tau}^{\ast}\in \chi_{\Omega}^{++}$ that is globally asymptotically stable for all $\tau\in(0,\infty)$.
\end{itemize}
\end{theorem}

\begin{proof}
(i)
By the definition of $\lambda_{p}(N_{\Omega}^{t})$, it is easy to deduce
\begin{equation}\label{501}
-\underset{x\in\bar{\Omega}}{\max} \{a(x,t)\}\leq\lambda_{p}(N_{\Omega}^{t})\leq-\underset{x\in\bar{\Omega}}
{\max} \bigg\{a(x,t)-\frac{\mu}{\sigma^{m}}\int_{\Omega}J_{\sigma}(x-y)dy\bigg\}.
\end{equation}
Thus, we have
\begin{equation*}
\int_{0}^{1}\lambda_{p}(N_{\Omega}^{t})dt\geq
\int_{0}^{1}-\underset{x\in\bar{\Omega}}{\max} \{a(x,t)\}>0.
\end{equation*}
By Corollary \ref{17}, we deduce the conclusion of part (i).

\medskip
(ii) Owing to the inequality \eqref{501}, we have
\begin{equation*}
\int_{0}^{1}\lambda_{p}(N_{\Omega}^{t})dt\leq\int_{0}^{1}-
\underset{x\in\bar{\Omega}}{\max} \bigg\{a(x,t)-\frac{\mu}{\sigma^{m}}\int_{\Omega}J_{\sigma}(x-y)dy\bigg\}dt\leq
-\int_{0}^{1}
\underset{x\in\bar{\Omega}}{\max} \{a(x,t)\}dt+\frac{\mu}{\sigma^{m}}<0.
\end{equation*}
By the definition of $\lambda_{p}(N_{\Omega})$, we obtain
\begin{equation}\label{502}
-\underset{x\in\bar{\Omega}}{\max} \{\hat{a}(x)\}\leq\lambda_{p}(N_{\Omega})\leq-\underset{x\in\bar{\Omega}}
{\max} \bigg\{\hat{a}(x)-\frac{\mu}{\sigma^{m}}\int_{\Omega}J_{\sigma}(x-y)dy\bigg\},
\end{equation}
which implies that
\begin{equation*}
\lambda_{p}(N_{\Omega})\geq-\underset{x\in\bar{\Omega}}{\max} \{\hat{a}(x)\}>0.
\end{equation*}
The conclusion of part (ii)
thus follows from 
Corollary \ref{17}.

\medskip
(iii) By the inequality \eqref{502}, it implies that
\begin{equation*}
\lambda_{p}(N_{\Omega})\leq
-\underset{x\in\bar{\Omega}}{\max} \bigg\{\hat{a}(x)-\frac{\mu}{\sigma^{m}}\int_{\Omega}J_{\sigma}(x-y)dy\bigg\}
\leq-\underset{x\in\bar{\Omega}}{\max} \{\hat{a}(x)\}+\frac{\mu}{\sigma^{m}}<0.
\end{equation*}
The conclusion of part (iii)
thus follows from 
Corollary \ref{17}.
\end{proof}

\subsection{Effects of the dispersal rate}
\noindent

This subsection is devoted to the proof of Theorem \ref{th105}. 
We recall the following result in \cite{shen-2019}:
\begin{lemma}\label{le301}
Assume that $f$ satisfies {\rm(F)}. If $\min_{\bar{\Omega}}\hat{a}>0$, then
for each $x\in\bar{\Omega}$, the equation
\begin{equation*}
\tau v_{t}=f(x,t,v)
\end{equation*}
has a unique positive 1-periodic solution, denoted by $v^{\ast}(x,t)$, which is continuous in x.
\end{lemma}

Although the proof of Theorem \ref{th105} (i) is similar to \cite[Theorem C]{shen-2019}, we still present the proof
for the convenience of the reader.

\begin{proof}[Proof of Theorem {\rm\ref{th105}}]
(i) By Theorem \ref{th102} (ii), there exists $\mu_{0}>0$ such that $\lambda_{p}(L_{\Omega}^{\tau,\mu,1,0})\leq-{\max_{\bar{\Omega}}
\hat{a}}/{2}<0$  for all $\mu\in(0,\mu_{0})$. Thus, it follows from Lemma \ref{th104}
that  (\ref{105}) admits a unique solution $u_{\mu}^{\ast}\in \chi_{\Omega}^{++}$ which is globally asymptotically stable for all $\mu\in(0,\mu_{0})$.

We claim that for each $0<\epsilon\ll 1$, there exists $\mu_{\epsilon}>0$ such that for each $\mu\in(0,\mu_{\epsilon})$, 
\begin{equation*}
v^{\ast}(x,t)-\epsilon\leq u_{\mu}^{\ast}(x,t)\leq v^{\ast}(x,t)+\epsilon,\ \ \ \ (x,t)\in \bar{\Omega}\times\mathbb{R}.
\end{equation*}
Let us prove the lower bound
only as the upper bound follows from similar arguments. Let $0<\epsilon\ll1$. By $\min_{\bar{\Omega}\times\mathbb{R}}v^{\ast}>0$, there exists $\delta=\delta(\epsilon)>0$ such that
\begin{equation*}
v(x,t):=(1-\delta)v^{\ast}(x,t)\geq v^{\ast}(x,t)-\epsilon>0,\ \ \ \ (x,t)\in \bar{\Omega}\times\mathbb{R}.
\end{equation*}
Note that for each $(x,t)\in \bar{\Omega}\times\mathbb{R}$,
\begin{align*}
&-\tau v_{t}(x,t)+\mu \int_{\Omega}J(x-y)(v(y,t)-v(x,t))dy+f(x,t,v(x,t))\\
=&-(1-\delta)\tau v_{t}^{\ast}(x,t)+(1-\delta)\mu \int_{\Omega}J(x-y)(v^{\ast}(y,t)-v^{\ast}(x,t))dy+f(x,t,v(x,t))\\
=&(1-\delta)\mu \int_{\Omega}J(x-y)(v^{\ast}(y,t)-v^{\ast}(x,t))dy
+f(x,t,v(x,t))-(1-\delta)f(x,t,v^{\ast}(x,t)).
\end{align*}
We see that as $\mu\rightarrow0^{+}$,
\begin{equation}\label{310}
(1-\delta)\mu \int_{\Omega}J(x-y)(v^{\ast}(y,t)-v^{\ast}(x,t))dy
\rightarrow0\ \ \ \ \text{uniformly in}\ (x,t)\in\bar{\Omega}\times\mathbb{R}.
\end{equation}
By (F)-(3), there holds for each $(x,t)\in\bar{\Omega}\times\mathbb{R}$,
\begin{equation*}
f(x,t,v(x,t))-(1-\delta)f(x,t,v^{\ast}(x,t))=v(x,t)\bigg[\frac
{f(x,t,v(x,t))}{v(x,t)}-\frac
{f(x,t,v^{\ast}(x,t))}{v^{\ast}(x,t)}\bigg]>0.
\end{equation*}
Thus, there exists $\mu_{\epsilon}>0$ such that for each $\mu\in(0,\mu_{\epsilon})$,
\begin{equation}\label{311}
\tau v_{t}(x,t)\leq\mu \int_{\Omega}J(x-y)(v(y,t)-v(x,t))dy+f(x,t,v(x,t))
\ \ \ \text{for all}\ (x,t)\in\bar{\Omega}\times\mathbb{R}.
\end{equation}

It remains to show that for each $\mu\in(0,\mu_{\epsilon})$, there holds
$v(x,t)\leq u_{\mu}^{\ast}(x,t)$ for all $(x,t)\in\bar{\Omega}\times\mathbb{R}$. To do so, let us fix any $\mu\in(0,\mu_{\epsilon})$ and define
\begin{equation*}
\alpha_{\ast}=\inf \{\alpha>0\ |\ v(x,t)\leq\alpha u_{\mu}^{\ast}(x,t)\ \ \text{for all}\ \ (x,t)\in\bar{\Omega}\times\mathbb{R} \}.
\end{equation*}
Since $u_{\mu}^{\ast}\in\chi_{\Omega}^{++}$ and $v$ is bounded, $\alpha_{\ast}$ is well-defined and positive. Thanks to the continuity
of $u_{\mu}^{\ast}$ and $v$, there holds $v(x,t)\leq\alpha_{\ast} u_{\mu}^{\ast}(x,t)$ and there exists $(x_{0},t_{0})\in\bar{\Omega}\times\mathbb{R}$ such that $v(x_{0},t_{0})=\alpha_{\ast} u_{\mu}^{\ast}(x_{0},t_{0})$.

If $\alpha_{\ast}\leq1$, then we are done. Therefore, let us assume $\alpha_{\ast}>1$. By the inequality (\ref{311}) and (F)-(3),
we see that $w(x,t):=v(x,t)-\alpha_{\ast} u_{\mu}^{\ast}(x,t)$ satisfies
\begin{equation}\label{312}
\begin{split}
\tau w_{t}(x,t)&\leq\mu \int_{\Omega}J(x-y)(w(y,t)-w(x,t))dy+f(x,t,v(x,t))-
\alpha_{\ast}f(x,t,u_{\mu}^{\ast}(x,t))\\
&<\mu \int_{\Omega}J(x-y)(w(y,t)-w(x,t))dy+f(x,t,v(x,t))-
f(x,t,\alpha_{\ast}u_{\mu}^{\ast}(x,t)).
\end{split}
\end{equation}
Hence, setting $(x,t)=(x_{0},t_{0})$ in (\ref{312}) yields
\begin{align*}
0=\tau w_{t}(x_{0},t_{0})<&\mu\int_{\Omega}J(x_{0}-y)
(w(y,t_{0})-w(x_{0},t_{0}))dy\\
&+f(x_{0},t_{0},v(x_{0},t_{0}))-
f(x_{0},t_{0},\alpha_{\ast}u_{\mu}^{\ast}(x_{0},t_{0}))\leq0,
\end{align*}
which is a contradiction. Hence, $\alpha_{\ast}\leq1$ and the proof is complete.

(ii) By Theorem \ref{th103}(i), there exists $\mu_{1}>0$ such that $\lambda_{p}(L_{\Omega}^{\tau,\mu,1,0})\leq-{\bar{\hat{a}}}/{2}<0$  for all $\mu\in(\mu_{1},\infty)$. Thus, it follows from Lemma \ref{th104}
that  (\ref{105}) admits a unique solution $u_{\mu}^{\ast}\in \chi_{\Omega}^{++}$ that is globally asymptotically stable for all $\mu\in(\mu_{1},\infty)$. Thus, we have
\begin{equation} \label{313}
\tau(u_{\mu}^{\ast})_{t}(x,t)=\mu\int_{\Omega}J(x-y)(u_{\mu}^{\ast}(y,t)-
u_{\mu}^{\ast}(x,t))dy+f(x,t,u_{\mu}^{\ast}(x,t)),~~~~~ (x,t)\in \bar{\Omega}\times\mathbb{R}.
\end{equation}
It is easy to check that there holds
\begin{equation*}
\underset{t\in[0,1]}{\sup}\|f(\cdot,t,u_{\mu}^{\ast}(\cdot,t))\|_{L^{\infty}
(\Omega)}\leq\underset{t\in[0,1]}{\sup}\|a(\cdot,t)u_{\mu}^{\ast}(\cdot,t)\|
_{L^{\infty}(\Omega)}\leq
M\underset{(x,t)\in\bar{\Omega}\times[0,1]}{\max}|a(x,t)|:=M A.
\end{equation*}

Multiplying (\ref{313}) by $u_{\mu}^{\ast}$ and integrating
over $\Omega\times(0,1)$ yield
\begin{align*}
\tau\int_{0}^{1}\int_{\Omega} (u_{\mu}^{\ast})_{t}u_{\mu}^{\ast} dxdt=&\int_{0}^{1}\int_{\Omega}f(x,t,u_{\mu}^{\ast}(x,t))u_{\mu}^{\ast}(x,t)
dxdt\\
&+\mu \int_{0}^{1}\int_{\Omega}\int_{\Omega}J(x-y) (u_{\mu}^{\ast}(y,t)-u_{\mu}^{\ast}(x,t))u_{\mu}^{\ast}(x,t)dydxdt.
\end{align*}
In view of $u_{\mu}^{\ast}\in \chi_{\Omega}^{++}$ and the symmetry of $J$,
we have
\begin{align*}
\frac{\mu}{2} \int_{0}^{1}\int_{\Omega}\int_{\Omega}J(x-y) (u_{\mu}^{\ast}(y,t)-u_{\mu}^{\ast}(x,t))^{2}dydxdt
=\int_{0}^{1}\int_{\Omega}f(x,t,u_{\mu}^{\ast}(x,t))u_{\mu}^{\ast}(x,t)
dxdt,
\end{align*}
which implies that
\begin{equation}\label{314}
\begin{split}
&\int_{0}^{1}\int_{\Omega}\int_{\Omega}J(x-y) (u_{\mu}^{\ast}(y,t)-u_{\mu}^{\ast}(x,t))^{2}dydxdt\\
=&\frac{2}{\mu}\int_{0}^{1}\int_{\Omega}f(x,t,u_{\mu}^{\ast}(x,t))u_{\mu}^{\ast}(x,t)
dxdt\\
\leq&\frac{2M^{2}A|\Omega|}{\mu}.
\end{split}
\end{equation}
Let us assume $U(x,t)=u_{\mu}^{\ast}(x,t)-\bar{u}(t)$, where $\bar{u}(t)=\frac{1}{|\Omega|}\int_{\Omega}u_{\mu}^{\ast}(x,t)dx$.
We get $\int_{\Omega}U(x,t)dx=0$. Integrating (\ref{313}) over
$\Omega$ and substituting
$U(x,t)=u_{\mu}^{\ast}(x,t)-\bar{u}(t)$ yield
\begin{equation*}
\begin{split}
\tau\bar{u}_{t}(t)&=\frac{1}{|\Omega|}
\int_{\Omega}f(x,t,U(x,t)+\bar{u}(t))dx\\
&=\frac{1}{|\Omega|}
\int_{\Omega}f(x,t,\bar{u}(t))dx+\frac{1}{|\Omega|}
\int_{\Omega}\big(f(x,t,U(x,t)+\bar{u}(t))-f(x,t,\bar{u}(t))\big)dx.
\end{split}
\end{equation*}
Owing to assumption (F), we have
\begin{equation*}
\begin{split}
&\bigg|\int_{\Omega}\big(f(x,t,U(x,t)+\bar{u}(t))-f(x,t,\bar{u}(t))\big)dx\bigg|\\
\leq&\bigg|\int_{\Omega}a(x,t)U(x,t)dx\bigg|\\
\leq&A|\Omega|^{\frac{1}{2}}\bigg(\int_{\Omega}U^{2}(x,t)dx\bigg)^{\frac{1}{2}}
\end{split}
\end{equation*}
By \cite[Page 1688, Formula (5.6)]{Shen-2015-DCDS}, there exists $C>0$ such that
\begin{equation}\label{315}
\int_{\Omega}U^{2}(x,t)dx\leq\frac{1}{2C}\int_{\Omega}\int_{\Omega}J(x-y) (U(y,t)-U(x,t))^{2} dydx~~~~~~~~\ \ \ \  \text{for all}\ \ \mu\gg1.
\end{equation}
By  (\ref{314}) and (\ref{315}), there holds
\begin{equation*}
\int_{0}^{1}\int_{\Omega}U^{2}(x,t)dxdt\leq\frac{M^{2}A|\Omega|}{C\mu},
\end{equation*}
which implies that
\begin{equation*}
\int_{0}^{1}\bigg|\int_{\Omega}\big(f(x,t,U(x,t)+\bar{u}(t))-f(x,t,\bar{u}(t))
\big)dx\bigg|dt=
O(\mu^{-\frac{1}{2}})\ \ \ \ \text{as} \ \ \mu\gg1.
\end{equation*}
Thus, we have
\begin{equation*}
\underset{\mu \rightarrow \infty }{\lim }\ u_{\mu}^{\ast}(x,t)=v^{\ast}(t) \ \ \text{uniformly in} \ \ (x,t)\in \bar{\Omega}\times\mathbb{R},
\end{equation*}
where $v^{\ast}(t)$ is the unique positive 1-periodic solution or zero solution of the  equation
\begin{equation*}
\tau v_{t}(t)=\frac{1}{|\Omega|}\int_{\Omega}f(x,t,v(t))dx.
\end{equation*}

Finally, we exclude that $v^{\ast}(t)$ is zero solution.
Divide (\ref{313}) by $u_{\mu}^{\ast}$ and integrate over $\Omega\times(0,1)$
to obtain
\begin{align*}
\tau\int_{0}^{1}\int_{\Omega} \frac{(u_{\mu}^{\ast})_{t}(x,t)}{u_{\mu}^{\ast}
(x,t)}dxdt=&\mu \int_{0}^{1}\int_{\Omega}\int_{\Omega}J(x-y) \frac{u_{\mu}^{\ast}(y,t)-u_{\mu}^{\ast}(x,t)}{u_{\mu}^{\ast}(x,t)}dydxdt\\
&+\int_{0}^{1}\int_{\Omega} \frac{f(x,t,u_{\mu}^{\ast}(x,t))}{u_{\mu}^{\ast}(x,t)}dxdt\ \ \ \  \text{for all}\ \ \mu\in(\mu_{1},\infty).
\end{align*}
Owing to $u_{\mu}^{\ast}\in \chi_{\Omega}^{++}$ and the symmetry of $J$,
we get
\begin{equation*}
\int_{0}^{1}\int_{\Omega} \frac{f(x,t,u_{\mu}^{\ast}(x,t))}{u_{\mu}^{\ast}(x,t)}dxdt=
-\frac{\mu}{2} \int_{0}^{1}\int_{\Omega}\int_{\Omega}J(x-y) \frac{(u_{\mu}^{\ast}(y,t)-u_{\mu}^{\ast}(x,t))^{2}}
{u_{\mu}^{\ast}(x,t)u_{\mu}^{\ast}(y,t)}dydxdt
\end{equation*}
for all $\mu\in(\mu_{1},\infty)$.
This implies that
\begin{equation*}
\int_{0}^{1}\int_{\Omega} \frac{f(x,t,u_{\mu}^{\ast}(x,t))}{u_{\mu}^{\ast}(x,t)}dxdt\leq0\ \ \ \  \text{for all}\ \ \mu\in(\mu_{1},\infty).
\end{equation*}
If $v^{*}(t)\equiv0$, then it follows from the above inequality and assumption
(F) that
\begin{equation*}
\int_{0}^{1}\int_{\Omega}a(x,t)dxdt\leq0,
\end{equation*}
which contradicts  $\bar{\hat{a}}>0$.
Therefore, $v^{\ast}$ is non-zero and the proof is complete.
\end{proof}
\subsection{Effects of the dispersal spread}
\noindent

This subsection is devoted to the proof of Theorem \ref{th106}.

\begin{proof}[Proof of Theorem {\rm\ref{th106}}]
(i) It follows from Theorem \ref{th102} and Lemma \ref{th104} that
there exists $\sigma_{1}>0$ such that  (\ref{105}) admits a unique solution $u_{\sigma}^{\ast}\in \chi_{\Omega}^{++}$ that is globally asymptotically stable for all $\sigma\in(\sigma_{1},\infty)$.
Similar to the proof of Theorem \ref{th105} (i), we have
\begin{equation*}
\underset{\sigma \rightarrow \infty }{\lim }\ u_{\sigma}^{\ast}(x,t)=v^{\ast}(x,t) \ \ \text{uniformly in} \ \ (x,t)\in \bar{\Omega}\times\mathbb{R}.
\end{equation*}


(ii) It follows from Theorem \ref{th103} and Lemma \ref{th104} that
there exists $\sigma_{0}>0$ such that  (\ref{105}) admits a unique solution $u_{\sigma}^{\ast}\in \chi_{\Omega}^{++}$ that is globally asymptotically stable for all $\sigma\in(0,\sigma_{0})$.
By the same proof of Theorem \ref{th105} (i), we can also obtain that
\begin{equation*}
\underset{\sigma \rightarrow 0^{+} }{\lim }\ u_{\sigma}^{\ast}(x,t)=v^{\ast}(x,t) \ \ \text{uniformly in} \ \ (x,t)\in \bar{\Omega}\times\mathbb{R}.
\end{equation*}
The proof is complete.
\end{proof}


\section*{Acknowledgments}
\noindent

YHS was supported by the China Scholarship Council.
WTL was partially supported by NSF of China (11731005, 11671180),
YL was partially supported by NSF (DMS-1853561)
and FYY was partially supported by NSF of China (11601205).
We thank Shuang Liu and Zhongwei Shen for helpful discussions.


\end{document}